\documentclass[11pt]{article}
\usepackage{graphicx}
\usepackage{epstopdf}
\usepackage{subcaption}
\usepackage{enumerate}
\usepackage{amssymb,a4wide,latexsym,makeidx,epsfig,fleqn}
\usepackage{amsthm}
\usepackage{amsmath}
\allowdisplaybreaks[4]
\usepackage{enumerate}
\usepackage{graphicx}
\usepackage{color}
\usepackage{epstopdf}
\newtheorem{theorem}{Theorem}[section]

\newtheorem{lemma}[theorem]{Lemma}

\begin{document}
	\textwidth 150mm \textheight 230mm
	\setlength{\topmargin}{-15mm}
	\title{Some sufficient conditions for a graph with minimum degree to be $k$-critical with respect to $[1,b]$-odd factors
		\footnote{This work is supported by the National Natural Science Foundations of China (Nos. 12371348, 12201258), the Postgraduate Research \& Practice Innovation Program of Jiangsu Normal University (No. 2025XKT0633).}}
	\author{{ Jiaxu Zhong, Yong Lu\footnote{Corresponding author}}\\
		{\small  School of Mathematics and Statistics, Jiangsu Normal University,}\\ {\small  Xuzhou, Jiangsu 221116,
			People's Republic
			of China.}\\
		{\small E-mails: JXZhong@163.com, luyong@jsnu.edu.cn}}
	
	\date{}
	\maketitle
	\begin{center}
		\begin{minipage}{120mm}
			\vskip 0.3cm
			\begin{center}
				{\small {\bf Abstract}}
			\end{center}
			{\small
				A graph $G$ is $k$-factor-critical if $G-S$ has a perfect matching for every subset $S \subseteq V(G)$ with $|S|=k$. A spanning subgraph $H$ of $G$ is called a $[1,b]$-odd factor if $b \equiv 1 \pmod{2}$ and $d_{H}(v) \in\left\lbrace 1, 3, \ldots, b\right\rbrace$ for every $v\in V(G),$ where $d_{H}(v)$ denotes the degree of vertex $v$ in $H$. Moreover, $G$ is said to be $k$-critical with respect to $[1,b]$-odd factors if $G-X$ contains a $[1,b]$-odd factor for every subset $X \subseteq V(G)$ with $|X|=k$. In this paper, we provide some sufficient conditions based on the distance spectral radius and the distance signless Laplacian spectral radius for a graph with minimum degree to be $k$-critical with respect to $[1,b]$-odd factors.

				 }
		\end{minipage}
	\end{center}

	\section{Introduction }
	\hspace{1.3em}
	Let $G=(V(G),E(G))$ be a finite, undirected and simple graph, where $V(G)$ is the vertex set and $E(G)$ is the edge set. We denote the \emph{order} and \emph{size} of $G$ by $|V(G)|=n$ and $|E(G)|=e(G)$, respectively. Let $d(v)$ be the \emph{degree} of vertex $v\in V(G)$, and $\delta(G)$ be the \emph{minimum degree} (or simply $\delta$) of $G$.
	For a vertex subset $S$ of $G$, we denote by $G-S$ and $G[S]$ the subgraph of $G$ obtained from $G$ by deleting the vertices in $S$ together with their incident edges and the subgraph of $G$ induced by $S$, respectively. The number of components and the number of odd components  of $G$ are denoted by $c(G)$ and $o(G)$. Let $K_{n}$ denote the \emph{complete graph} of order $n$.
	For two vertex-disjoint graphs $G_{1}$ and $G_{2}$, we use $G_{1}\cup G_{2}$ to denote the \emph{disjoint union} of $G_{1}$ and $G_{2}$.
	The \emph{join} $G_{1}\vee G_{2}$ is the graph obtained from $G_{1} \cup G_{2}$ by adding all possible edges between $V(G_{1})$ and $V(G_{2})$. A graph $G$ of order $n$ is called \emph{$k$-connected} if $n>k$ and $G-X$ is connected for every set $X\subseteq V(G)$ with $|X|<k$, where $k$ is a positive integer.
	
	For a simple graph $G$ of order $n$, its \emph{adjacency matrix} $A(G)=(a_{ij})_{n\times n}$ is a symmetric matrix with $a_{ij}=1$ if and only if vertices $v_{i}$ and $v_{j}$ are adjacent, and $a_{ij}=0$ otherwise. The \emph{spectral radius} $\rho(G)$ of $G$ is the largest eigenvalue of $A(G)$. The \emph{signless Laplacian matrix} of $G$ is defined as $Q(G)=\hat{D}+A(G)$, where $\hat{D}$ denotes the \emph{diagonal degree matrix} of $G$. The \emph{signless Laplacian spectral radius} $q(G)$ is the largest eigenvalue of $Q(G)$. For vertices $v_{i}, v_{j} \in V(G)$, the \emph{distance} between $v_{i}$ and $v_{j}$, denoted $d_{G}(v_{i}, v_{j})$ (or simply $d_{ij}$), is the length of the shortest path connecting $v_{i}$ and $v_{j}$. The \emph{distance matrix} $D(G)=(d_{ij})_{n\times n}$ of $G$ is a symmetric matrix whose $(i,j)$-entry is $d_{ij}$, and $\mu_{1}(G)\geq\mu_{2}(G)\geq\cdots\geq\mu_{n}(G)$ are the eigenvalues of $D(G)$. The largest eigenvalue $\mu_{1}(G)$ is called the \emph{distance spectral radius} of $G$.
	The \emph{transmission} $Tr(v)$ of a vertex $v \in V(G)$ is the sum of distances from $v$ to all vertices in $G$.
	A graph is \emph{$k$-transmission-regular} (or simply \emph{transmission-regular}) if the rows of its distance matrix each have the constant sum $k$.
	Aouchiche and Hansen \cite{AHP} introduced the \emph{distance signless Laplacian matrix}  $Q_{D}(G)=D(G)+Tr(G)$ of $G$, where $Tr(G)$ is the diagonal matrix whose diagonal entries are the vertex transmissions in $G$. Let $\eta_{1}(G)\geq\eta_{2}(G)\geq\cdots\geq\eta_{n}(G)$ be eigenvalues of $Q_{D}(G)$.
	The largest eigenvalue $\eta_{1}(G)$ is called the \emph{distance signless Laplacian spectral radius} of $G$.

	A \emph{matching} in a graph $G$ is a subset $M$ of $E(G)$ such that no vertex of $G$ is incident with more than one edge in $M$. A \emph{perfect matching} in $G$ is a matching such that every vertex of $G$ is incident with precisely one edge in it. The problem of determining when a general graph has a perfect matching is a classic topic in graph theory. In 1917, Frobenius \cite{CL} initiated the characterization of perfect matchings, who showed that an $n$-vertex bipartite graph has a perfect matching if and only if the cardinality of each vertex cover is at least $\frac{n}{2}$. In 1947, Tutte \cite{HDZ} published his celebrated theorem, which provided a necessary and sufficient condition for a general graph to have a perfect matching. Anderson \cite{AH} investigated the existence of perfect matchings in graphs using neighborhood conditions. Since then, many researchers have been interested in finding sufficient conditions to guarantee the existence of a perfect matching in a graph using various graph invariants, see \cite{EJK, FLL, LFS}.
	
	On the other hand, researchers have been interested in studying graphs where every subgraph of a given order has a perfect matching. Favaron \cite{C} and Yu \cite{MK} independently introduced the concept of $k$-factor-critical graphs: a graph $G$ of order $n$ is \emph{$k$-factor-critical} (where $0 \leq k < n$) if removing any $k$ vertices leaves a graph with a perfect matching. Clearly, if $G$ is a $k$-factor-critical graph with $n$ vertices, then $k$ and $n$ have the same parity. Fan and Lin \cite{BY} gave an adjacency spectral condition for a connected graph with minimum degree to be $k$-factor-critical. Zheng et al. \cite{ZLLW} established three sufficient conditions based on the size, the signless Laplacian spectral radius and the distance spectral radius of a connected graph to be $k$-factor-critical. Zhou and Zhang \cite{ZZL} derived a condition on the signless Laplacian spectral radius for the existence of $2k$-factor-critical graphs. More results on the relationships between the spectral radius and spanning subgraphs can be found in \cite{XZL, ZC, ZZL, ZSZ}.

	Let $g$ and $f$ be two positive integer-valued functions defined on $V(G)$ such that $g(x)\leq f(x)$ holds for any $x\in V(G)$. Then a spanning subgraph $F$ of $G$ is called a \emph{$(g,f)$-factor} if $g(v)\leq d_{F}(v)\leq f(v)$ holds for any $v\in V(G)$. Let $a$ and $b$ be two positive integers with $a\leq b$. A $(g,f)$-factor is called an \emph{$[a,b]$-factor} if $g(v)=a$ and $f(v)=b$ for any $v\in V(G)$. A spanning subgraph $F$ of $G$ is called a \emph{$(1,f)$-odd factor} if $f(v) \equiv 1 \pmod{2}$ and $d_{F}(v)\in\left\lbrace 1,3,\ldots,f(v)\right\rbrace$ for every $v\in V(G)$. A $(1,f)$-odd factor is called a \emph{$[1,b]$-odd factor} if $f(v)=b$ (an odd integer) for every $v\in V(G)$. In fact, a perfect matching is a special $[1,b]$-odd factor when $b=1$. Cui and Kano \cite{B} provided a sufficient condition for a graph to admit a $\{1, 3, \ldots, 2n-1\}$-factor. Fan, Lin and Lu \cite{BH} provided some spectral conditions for the existence of a $[1,b]$-odd factor in a connected graph with minimum degree, as well as for the existence of an $[a,b]$-factor. Zhou and Liu \cite{ZW} established a spectral radius condition for a connected graph to have a $[1,b]$-odd factor, and further derived three lower bounds on the edge number of even-order graphs that guarantee the existence of such a factor. Additional results on graph factors can be found in \cite{CLX, CFL, XZL}.

	For an integer $k\geq1$, a graph $G$ of order $n\geq k+2$ is said to be \emph{$k$-critical with respect to $(1,f)$-odd factors} if for any subset $X\subseteq V(G)$ with $|X|=k$, the graph $G-X$ contains a $(1,f)$-odd factor. Similarly, $G$ is \emph{$k$-critical with respect to $[1,b]$-odd factors} if $G-X$ contains a $[1,b]$-odd factor for every $X\subseteq V(G)$ with $|X|=k$.

	Zhou \cite{YYSX} established conditions on the size and the spectral radius for a graph with a given minimum degree to be $k$-critical with respect to $[1,b]$-odd factors. Wang et al. \cite{LL} presented a sufficient condition based on the signless Laplacian spectral radius to guarantee that a graph with a given minimum degree is $k$-critical with respect to $[1,b]$-odd factors.
	Wang and Zhang \cite{JL} considered the above problem in a connected graph of order $n$ in terms of the distance spectral radius.
	
	First, we restate two sufficient conditions from Zhou \cite{YYSX} that ensure a graph $G$ with a given minimum degree is $k$-critical with respect to $[1,b]$-odd factors, based on its size and the spectral radius of $G$.
	
	\noindent\begin{theorem}\label{th:1.1.} \cite{YYSX}
		Let $b$, $k$ and $n$ be three positive integers with $n\equiv k\pmod{2}$ and $b\equiv1\pmod{2}$, and let $G$ be a $(k+1)$-connected graph of order
		\[
		n\geq \max \left\lbrace
		\begin{array}{cc}
			\frac{b^2k^2-2b^2k\delta -4b^2k-bk+b^2\delta^2+4b^2\delta+7b\delta+b^2+8b+1}{6b},~(b+5)\delta-(b+4)k-b+1+\frac{5}{b}
		\end{array}
		\right\}
		\]
		with minimum degree $\delta$. If
		$$e(G)\geq e(K_{\delta}\vee(K_{n-(b+1)\delta+bk-1}\cup(b\delta-bk+1)K_{1})),$$ then $G$ is $k$-critical with respect to $[1,b]$-odd factors, unless
		$$G=K_{\delta}\vee(K_{n-(b+1)\delta+bk-1}\cup(b\delta-bk+1)K_{1}).$$
	\end{theorem}
	For clarity, the exceptional graph in Theorem \ref{th:1.1.} is shown in Figure 1.
	\begin{figure}[htbp]
		\centering
	    \includegraphics[scale=0.9]{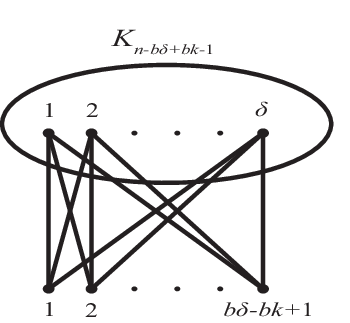}
		\caption{$K_{\delta}\vee(K_{n-(b+1)\delta+bk-1}\cup(b\delta-bk+1)K_{1})$.}
	\end{figure}
	
	\noindent\begin{theorem}\label{th:1.2.} \cite{YYSX}
		Let $b$, $k$ and $n$ be three positive integers with $n\equiv k\pmod{2}$ and $b\equiv1\pmod{2}$, and let $G$ be a $(k+1)$-connected graph of order $n\geq \max \{b\delta^2-bk, ~(2b+3)\delta-bk+1\}$ with minimum degree $\delta$. If
		$$\rho(G)\geq \rho(K_{\delta}\vee(K_{n-(b+1)\delta+bk-1}\cup(b\delta-bk+1)K_{1})),$$
		then $G$ is $k$-critical with respect to $[1,b]$-odd factors, unless
		$$G= K_{\delta}\vee(K_{n-(b+1)\delta+bk-1}\cup(b\delta-bk+1)K_{1}).$$
	\end{theorem}
	Secondly, we recall a sufficient condition established by Wang et al. \cite{LL} based on the signless Laplacian spectral radius for a graph with a given minimum degree to be $k$-critical with respect to $[1,b]$-odd factors.
	
	\noindent\begin{theorem}\label{th:1.3.} \cite{LL}
		Let $b$, $k$ and $n$ be positive integers such that $n \equiv k \pmod{2}$ and $b \equiv 1 \pmod{2}$, and $n\geq (2b+4.3)\delta-2bk+1.1$. Let $G$ be a $(k+1)$-connected graph of order $n$ with minimum degree $\delta$. If
		$$q(G)\geq q(K_{\delta}\vee(K_{n-(b+1)\delta+bk-1}\cup(b\delta-bk+1)K_{1})),$$
		then $G$ is $k$-critical with respect to $[1,b]$-odd factors, unless $$G\cong K_{\delta}\vee(K_{n-(b+1)\delta+bk-1}\cup(b\delta-bk+1)K_{1}).$$
	\end{theorem}
	
	Finally, we recall a sufficient condition established by Wang and Zhang \cite{JL} based on the distance spectral radius for a connected graph to be $k$-critical with respect to $[1,b]$-odd factors.
	
	\noindent\begin{theorem}\label{th:1.4.} \cite{JL}
		Let $b$, $k$ and $n$ be three positive integers with $b \equiv 1 \pmod{2}$ and $n \equiv k\pmod{2}$, and let $G$ be a connected graph of order $n\geq\frac{b^2+2bk+5b+2k+4}{b}$. If $G$ satisfies
		$$\mu_{1}(G)\leq\mu_{1}(K_{k+1}\vee(K_{n-b-k-2}\cup(b+1)K_{1})),$$
		then $G$ is $k$-critical with respect to $[1,b]$-odd factors, unless
		$$G\in\left\lbrace K_{k}\vee \left(K_{n-k-1}\cup K_{1}\right),~K_{k+1}\vee \left(K_{n-k-b-2}\cup (b+1)K_{1}\right)\right\rbrace.$$
	\end{theorem}
	
	Motivated by \cite{JL, LL, YYSX}, we study the problem of the existence of $k$-critical graphs with respect to $[1,b]$-odd factors, and we present a different sufficient condition on the distance spectral radius to ensure that a graph $G$ with minimum degree is $k$-critical with respect to $[1,b]$-odd factors. We state our main result as follows:

	\noindent\begin{theorem}\label{th:1.5.}
		Let $b$, $k$ and $n$ be positive integers with $n \equiv k \pmod{2}$ and $b \equiv 1 \pmod{2},$ and let $G$ be a $(k+1)$-connected graph of order $n\geq \max \{(2b^2+3b+11)\delta-(2b^2+\frac{5}{2}b+\frac{3}{2})k+\frac{3}{2}b+2+\frac{3}{2b},~\frac{2}{3}b^2\delta^3+\frac{4}{3}b^2k\delta^2\}$ with minimum degree $\delta$. If
		$$\mu_{1}(G)\leq\mu_{1}(K_{\delta}\vee(K_{n-(b+1)\delta+bk-1}\cup(b\delta-bk+1)K_{1})),$$
		then $G$ is $k$-critical with respect to $[1,b]$-odd factors, unless $$G\cong K_{\delta}\vee(K_{n-(b+1)\delta+bk-1}\cup(b\delta-bk+1)K_{1}).$$
	\end{theorem}

	Let $b$, $k$ and $n$ be three positive integers with $n \equiv k \pmod{2}$ and $b \equiv 1\pmod{2}$, and let $G^{\prime}:= K_{\delta}\vee(K_{n-(b+1)\delta+bk-1}\cup(b\delta-bk+1)K_{1})$. If $\delta=k+2$, then $G^{\prime}= K_{k+2}\vee(K_{n-2b-k-3}\cup(2b+1)K_{1})$. It is easy to observe that $K_{k+2}\vee(K_{n-2b-k-3}\cup(2b+1)K_{1})$ is a proper subgraph of the extremal graph $K_{k+1}\vee(K_{n-b-k-2}\cup(b+1)K_{1})$ as discussed in Theorem 1.4.
	Let $G^{\ast}$ be the graph obtained from $K_{k+2}\vee(K_{n-2b-k-3}\cup(2b+1)K_{1})$ by adding one independent edge between two vertices in $(2b+1)K_{1}.$ See (Fig.2), $G^{\ast}$ is also a proper subgraph of $K_{k+1}\vee(K_{n-b-k-2}\cup(b+1)K_{1})$. We further obtain that  $\mu_{1}(K_{k+1}\vee(K_{n-b-k-2}\cup(b+1)K_{1}))<\mu_{1}(G^{\ast})<\mu_{1}(K_{k+2}\vee(K_{n-2b-k-3}\cup(2b+1)K_{1})).$ By Theorem 1.1, we know $G^{\ast}$ is beyond the conditions of Theorem 1.4, yet it is still $k$-critical with respect to $[1,b]$-odd factors. This is also the reason why we proposed Theorem 1.5.
	
	\begin{figure}[htbp]
		\centering
         \includegraphics[scale=0.9]{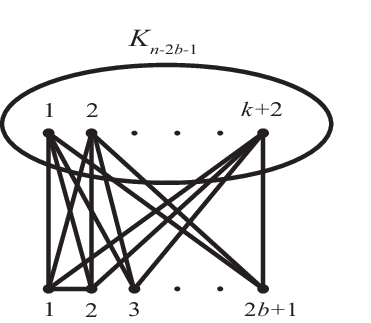}
		\caption{$G^{\ast}$.}
	\end{figure}
	
	Our second main result gives a sufficient condition to ensure that a graph $G$ with minimum degree is $k$-critical with respect to $[1,b]$-odd factors based on the distance signless Laplacian spectral radius of $G$.
	
	\noindent\begin{theorem}\label{th:1.6.}
		Let $b$, $k$ and $n$ be positive integers with $n \equiv k \pmod{2}$, $b \equiv 1 \pmod{2}$ and $b\geq k$, and let $G$ be a $(k+1)$-connected graph of order $n \geq \max \{(2b^2+4b)\delta^2+2\delta +2b^2k^2,~\frac{6}{5}b^2\delta^3+\frac{8}{5}b^2k\delta^2\}$ with minimum degree $\delta$. If
		$$\eta_{1}(G)\leq\eta_{1}(K_{\delta}\vee(K_{n-(b+1)\delta+bk-1}\cup(b\delta-bk+1)K_{1})),$$
		then $G$ is $k$-critical with respect to $[1,b]$-odd factors, unless $$G\cong K_{\delta}\vee(K_{n-(b+1)\delta+bk-1}\cup(b\delta-bk+1)K_{1}).$$
	\end{theorem}

	Kano and Matsuda \cite{CLX} established the following connectivity property for graphs that are $k$-critical with respect to $(1, f)$-odd factors.
	\noindent\begin{theorem}\label{th:1.7.} \cite{CLX}
		Let $k$ be a positive integer, and let $G$ be a graph of order $n\geq k+2$ such that $n \equiv k \pmod{2}$. If $G$ is $k$-critical with respect to $(1,f)$-odd factors, then $G$ is $k$-connected.
	\end{theorem}
	
	In the proof of Theorems 1.5 and 1.6, we assume that $G$ is $(k+1)$-connected, which leaves the question open of whether these results still hold for graphs with exactly connectivity  $k$.
	
	The rest of this paper is organized as follows: In Section 2, some lemmas used in this paper are presented. In Section 3, we present the proof of Theorem 1.5. In Section 4, we give the proof of Theorem 1.6.
	
	\section{Preliminaries }
	\hspace{1.3em}
	In this section, we first provide a necessary and sufficient condition for a graph to be $k$-critical with respect to $(1,f)$-odd factors. Cui and Kano \cite{B} confirmed the following result for the case $k=0$, while Kano and Matsuda \cite{CLX} later extended this result to positive integers $k$.
	\noindent\begin{lemma}\label{le:2.1.} \cite{B, CLX}
		Let $k$ be a nonnegative integer, and let $G$ be a graph of order $n\geq k+2$. Then $G$ is $k$-critical with respect to $(1,f)$-odd factors if and only if $$o(G-S) \leq \sum\limits_{v\in S}f(v)- \max\left\lbrace \sum\limits_{v\in X}f(v): X\subseteq S, |X|=k\right\rbrace$$
		for every subset $S\subseteq V(G)$ with $|S|\geq k$, where $o(G-S)$ denotes the number of odd components in $G-S$.
	\end{lemma}
	Next, we explain the concepts of \emph{quotient matrix} and \emph{equitable partition}. Let $M$ be a real $n\times n$ matrix  and let $X=\{1,2,\ldots,n\}$. Given a partition  $\Pi=\{X_{1}, X_{2}, \ldots, X_{m}\}$ with $X=X_{1}\cup X_{2}\cup\cdots \cup X_{m}$, the matrix $M$ can be written in the following block form
	\begin{align*}
		M=\left(
		\begin{array}{ccccccccc}
			M_{11} & M_{12} & \cdots & M_{1m} \\
			M_{21} & M_{22} & \cdots & M_{2m} \\
			\vdots & \vdots & \ddots &\vdots \\
			M_{m1} & M_{m2} & \cdots & M_{mm} \\
		\end{array}
		\right).
	\end{align*}
	The quotient matrix $R(M)$ of the matrix $M$ (with respect to the given partition) is the $m\times m$ matrix whose entries are the average row sums of the blocks $M_{i,j}$ of $M$. The above partition is called equitable if each block $M_{i,j}$ of $M$ has constant row sum.
	
	\noindent\begin{lemma}\label{le:2.2.} \cite{MH}
		Let $M$ be a real matrix of order $n$ with an equitable partition $\Pi$, and let $R(M)$ be the corresponding equitable quotient matrix. Then the eigenvalues of $R(M)$ are also eigenvalues of $M$. Furthermore, if $M$ is nonnegative, then $\rho(R(M))=\rho(M)$, where $\rho(R(M))$ and $\rho(M)$ denote the largest eigenvalues of the matrices $R(M)$ and $M$.
	\end{lemma}
	
	\noindent\begin{lemma}\label{le:2.3.} \cite{LL}
		Let $M$ be a nonnegative irreducible matrix of order $n$ with an equitable partition $\Pi$, and let $X$ be the Perron vector of $M$. Then the entries of $X$ are constant on each cell of the partition $\Pi$.
	\end{lemma}
	
	We present the following two fundamental results about the distance spectral radius and the distance signless Laplacian spectral radius of a graph.
	\noindent\begin{lemma}\label{le:2.4.} \cite{CLW}
		Let $G$ be a connected graph with two nonadjacent vertices $u, v\in V(G)$. Then $\mu_{1}(G+uv) < \mu_{1}(G)$.
	\end{lemma}
	\noindent\begin{lemma}\label{le:2.5.} \cite{CLW}
		Let $e$ be an edge of a graph $G$ such that $G-e$ is still connected. Then $\eta_{1}(G-e) > \eta_{1}(G)$.
	\end{lemma}

	\noindent\begin{lemma}\label{le:2.6.} \cite{CLW}
		Let $A$ be a real symmetric $n\times n$ matrix and let $B$ be an $m\times m$ principal submatrix of $A$ with $m<n$. Let $\lambda_{1}(A) \geq \lambda_{2}(A) \geq \cdots \geq \lambda_{n}(A)$ be the eigenvalues of $A$, and $\lambda_{1}(B) \geq \lambda_{2}(B) \geq \cdots \geq \lambda_{m}(B)$ be the eigenvalues of $B$. Then for $i=1, 2, \ldots, m$,  $$\lambda_{i}(A) \geq \lambda_{i}(B) \geq \lambda_{n-m+i}(A).$$
	\end{lemma}
	The \emph{Wiener index} $W(G)$ of a connected graph $G$ of order $n$ is defined by the sum of all distances in $G$, that is, $W(G)=\sum\limits_{i<j}d_{ij}(G)$.
	Xing et al. \cite{LLS} presented a lower bound on the distance signless Laplacian spectral radius of a graph.
	
	\noindent\begin{lemma}\label{le:2.7.} \cite{LLS}
		Let $G$ be a connected graph of order $n$.
		Then $$\eta_{1}(G) \geq \frac{4W(G)}{n}$$
		with equality if and only if $G$ is transmission-regular.
	\end{lemma}
	
	Finally, we present some relationships between the magnitudes of the distance spectral radius and the distance signless Laplacian spectral radius among several graphs.
	
	\noindent\begin{lemma}\label{le:2.8.} \cite{ZLLW}
		Let $n=s+\sum\limits_{i=1}^{t}n_{i}$. If $n_{1}\geq n_{2}\geq\cdots\geq n_{t}\geq p\geq1$ and $n_{1}<n-s-p(t-1)$, then
		$\mu_{1}(K_{s}\vee (K_{n_{1}}\cup K_{n_{2}}\cup\cdots\cup K_{n_{t}}))>\mu_{1}(K_{s}\vee (K_{n-s-p(t-1)}\cup (t-1)K_{p}))$.
	\end{lemma}

	\noindent\begin{lemma}\label{le:2.9.} \cite{LLC}
		Let $n$, $t$, $s$ and $n_{i} (i=1,2,\ldots,t)$ be positive integers with $n_{1}\geq n_{2}\geq\cdots\geq n_{t}\geq1$ and $n=s+\sum\limits_{i=1}^{t}n_{i}$.
		Then
		$$\eta_{1}(K_{s}\vee (K_{n_{1}}\cup K_{n_{2}}\cup\cdots\cup K_{n_{t}}))\geq\eta_{1}(K_{s}\vee (K_{n-s-t+1}\cup(t-1)K_{1}))$$
		with equality if and only if $K_{s} \vee (K_{n_{1}} \cup K_{n_{2}} \cup \cdots \cup K_{n_{t}}) \cong K_{s} \vee (K_{n-s-t+1} \cup (t-1)K_{1})$.
	\end{lemma}

	\noindent\begin{lemma}\label{le:2.10.} \cite{LLJ}
		Let $n$, $p$, $s$, $t$ and $n_{i}(i=1,2,\ldots,t)$ be positive integers with $n_{1}\geq5p$, $t\geq s+1$, $n_{1}\geq n_{2}\geq\cdots\geq n_{t}\geq p\geq1$ and $n=s+\sum\limits_{i=1}^{t}n_{i}$.
		Then
		$$\eta_{1}(K_{s}\vee (K_{n_{1}}\cup K_{n_{2}}\cup\cdots\cup K_{n_{t}}))\geq\eta_{1}(K_{s}\vee (K_{n-s-p(t-1)}\cup(t-1)K_{p}))$$
		with equality if and only if $K_{s}\vee (K_{n_{1}}\cup K_{n_{2}}\cup\cdots\cup K_{n_{t}})\cong K_{s}\vee (K_{n-s-p(t-1)}\cup(t-1)K_{p})$.
	\end{lemma}

	\section{Proof of Theorem \ref{th:1.5.}}
	\hspace{1.3em}
	
	In this section, we will give the proof of Theorem 1.5.\\
	
	\noindent\textbf{Proof of Theorem \ref{th:1.5.}.}

	Suppose that $G$ is a $(k+1)$-connected graph that is not $k$-critical with respect to $[1,b]$-odd factors. Then by Lemma \ref{le:2.1.} (with $f=b$), there exists a subset $S\subseteq V(G)$ such that $|S|=s\geq k$ satisfying $o(G-S)\geq bs - bk + 1$. Since $k\equiv n\equiv s+o(G-S)\pmod{2}$ and $b\equiv 1\pmod{2}$, it follows that $s-k\equiv o(G-S)\pmod{2}$ and $s-k\equiv b(s-k)\pmod{2}.$ Thus, $o(G-S)\geq b(s-k)+2$. This implies $$n\geq s+o(G-S)\geq s+b(s-k)+2\Longrightarrow s\leq\frac{n+bk-2}{b+1}.$$ Then $G$ is a spanning subgraph of $G_1=K_{s}\vee(K_{n_{1}}\cup K_{n_{2}}\cup\cdots\cup K_{n_{bs-bk+2}})$ for some positive odd integers $n_{1}\geq n_{2}\geq\cdots\geq n_{bs-bk+2}$ and $\sum\limits_{i=1}^{bs-bk+2}n_{i}=n-s$.
	By Lemma \ref{le:2.4.}, we obtain
	\begin{align*}
		\mu_{1}(G)\geq\mu_1(G_1)
	\end{align*}
	with equality if and only if $G\cong G_1$.
	If $s=k$, then $o(G-S)\geq 2$, contradicting the $(k+1)$-connectivity of $G$. Thus, $s\geq k+1$. We proceed by considering the following three possible cases.\\

	\textbf{Case 1.} $s\geq\delta+1$.

	Let $G_2=K_s\vee(K_{n-(b+1)s+bk-1}\cup(bs-bk+1)K_1)$. By Lemma \ref{le:2.8.}, we have
	\begin{align*}
		\mu_1(G_1)\geq\mu_1(G_2)
	\end{align*}
	with equality if and only if $(n_1,n_2,\ldots,n_{bs-bk+2})=(n-(b+1)s+bk-1, 1, \ldots, 1)$.

	Consider the partition $V(G_2)=V(K_s)\cup V(K_{n-(b+1)s+bk-1})\cup V((bs-bk+1)K_1)$, the corresponding quotient matrix of $D(G_2)$ equals
	\begin{align*}
		\left(\begin{matrix}	
			s-1 & n-(b+1)s+bk-1 & bs-bk+1 \cr
			s & n-(b+1)s+bk-2 & 2(bs-bk+1) \cr
			s & 2(n-(b+1)s+bk-1) & 2(bs-bk) \cr
		\end{matrix}\right).
	\end{align*}
	Then we obtain that the characteristic polynomial of the corresponding quotient matrix of $D(G_2)$ is given by $g_{D(G_2)}(x)$, where
	\begin{align*}
		g_{D(G_2)}(x)=&x^3+(3+bk-n-bs)x^2+[(2b^2+3b)s^2-(2bn+4b^2k+3bk-3b-3)s+(2bk
		\\&-5)n+2b^2k^2-3bk+6]x-(b^2+b)s^3+(bn+2b^2k+bk+2b^2+b-1)s^2
		\\&-(bkn+2bn-n+b^2k^2+4b^2k+bk-4b-2)s+(2bk-4)n+2b^2k^2-4bk+4.
	\end{align*}
	Since the partition $V(G_2)=V(K_s)\cup V(K_{n-(b+1)s+bk-1})\cup V((bs-bk+1)K_1)$ is equitable, by Lemma \ref{le:2.2.}, the largest root of $g_{D(G_2)}(x)=0$ equals $\mu_{1}(G_2)$.
	
	Let $G^{\prime}=K_\delta \vee(K_{n-(b+1)\delta+bk-1}\cup(b\delta -bk+1)K_1)$.
	Now consider the partition $V(G^{\prime})=V(K_\delta)\cup V(K_{n-(b+1)\delta+bk-1})\cup V((b\delta-bk+1)K_1)$, the corresponding quotient matrix of $D(G^{\prime})$ equals
	\begin{align*}
		\left(\begin{matrix}	
			\delta-1 & n-(b+1)\delta+bk-1 & b\delta-bk+1 \\
			\delta & n-(b+1)\delta+bk-2 & 2(b\delta-bk+1) \\
			\delta & 2(n-(b+1)\delta+bk-1) & 2(b\delta-bk) \\
		\end{matrix}\right).
	\end{align*}
	Then the characteristic polynomial of the corresponding quotient matrix of $D(G^{\prime})$ is given by $g_{D(G^{\prime})}(x)$, where
	\begin{align*}
		g_{D(G^{\prime})}(x)=&x^3+(3+bk-n-b\delta)x^2+[(2b^2+3b)\delta^2-(2bn+4b^2k+3bk-3b-3)\delta+(2bk
		\\&-5)n+2b^2k^2-3bk+6]x-(b^2+b)\delta^3+(bn+2b^2k+bk+2b^2+b-1)\delta^2
		\\&-(bkn+2bn-n+b^2k^2+4b^2k+bk-4b-2)\delta+(2bk-4)n+2b^2k^2-4bk+4.
	\end{align*}
	Since the partition $V(G^{\prime})=V(K_\delta)\cup V(K_{n-(b+1)\delta+bk-1})\cup V((b\delta-bk+1)K_1)$ is equitable, by Lemma \ref{le:2.2.}, the largest root of $g_{D(G^{\prime})}(x)=0$ equals $\mu_{1}(G^{\prime})$. By Lemma \ref{le:2.6.}, we have $\mu_{1}(G^{\prime})\geq\mu_{1}(K_{n-b\delta+bk-1})=n-b\delta+bk-2$.

	In order to prove $\mu_{1}(G_{2})>\mu_{1}(G^{\prime})$, it suffices to show $g_{D(G_2)}(\mu_{1}(G^{\prime}))<0$. Note that $g_{D(G^{\prime})}(\mu_{1}(G^{\prime}))=0$. Hence,
	\begin{align*}
		g_{D(G_2)}(\mu_1(G^{\prime}))=&g_{D(G_2)}(\mu_1(G^{\prime}))-g_{D(G^{\prime})}(\mu_1(G^{\prime}))
		\\=&(\delta-s)[ b\mu_{1}^{2}(G^{\prime})+(b(2n+3k-3(s+\delta+1))-2b^2(s+\delta-2k)-3)
		\\&\cdot \mu_{1}(G^{\prime})+n(bk+2b-1)+(bk-1)(bk+4b+2)
		\\&-(s+\delta)(b(n+k-\delta+1)+b^2(2k-\delta+2)-1)+b(1+b)s^2].
	\end{align*}
	Let $g(x)=bx^2+[b(2n+3k-3(s+\delta+1))-2b^2(s+\delta-2k)-3]x
	+n(bk+2b-1)+(bk-1)(bk+4b+2)-(s+\delta)[b(n+k-\delta+1)+b^2(2k-\delta+2)-1]+b(1+b)s^2$
	be a real function in $x$ with $x\in[n-b\delta+bk-2,+\infty)$. Then the derivative function of $g(x)$ is
	\begin{align*}
		g^{\prime}(x)=2bx+b(2n+3k-3(s+\delta+1))-2b^2(s+\delta-2k)-3.
	\end{align*}
	Hence, $\frac{2b^2(s+\delta-2k)+3-b(2n+3k-3(s+\delta+1))}{2b}$ is the unique solution of $g^{\prime}(x)=0$.
	\\Together with
	$n\geq s+b(s-k)+2$ and $n\geq (2b^2+3b+11)\delta-(2b^2+\frac{5}{2}b+\frac{3}{2})k+\frac{3}{2}b+2+\frac{3}{2b},$
	we have
	\begin{align*}
		&(n-b\delta+bk-2)-\frac{2b^2(s+\delta-2k)+3-b[ 2n+3k-3(s+\delta+1)]}{2b}
		\\=&\frac{4bn+(6b^2+3b)k-(2b^2+3b)s-(4b^2+3b)\delta-7b-3}{2b}
		\\\geq& \frac{bn+(6b^2+3b)k-(2b^2+3b)s-(4b^2+3b)\delta-3-7b+3bs+3b^2s-3b^2k+6b}{2b}
		\\=&\frac{bn+(3b^2+3b)k-(4b^2+3b)\delta+b^2s-b-3}{2b}
		\\\geq& \frac{(2b^3+3b^2+11b)\delta-(2b^3+\frac{5}{2}b^2+\frac{3}{2}b)k+\frac{3}{2}b^2+2b+\frac{3}{2}+(3b^2+3b)k-(4b^2+3b)\delta+b^2}{2b}\nonumber
		\\&\frac{+b^2-b-3}{2b}
		\\=&\frac{(2b^3+8b)\delta-(2b^3-\frac{1}{2}b^2-\frac{3}{2}b)k+\frac{5}{2}b^2+b-3}{2b}
		\\\geq&\frac{\frac{1}{2}b^2k+2b^3+\frac{5}{2}b^2-\frac{1}{2}b-\frac{3}{2}}{2b}
		\\>&0.
	\end{align*}
	Consequently, $g(x)$ is monotonically increasing in the interval $[n-b\delta+bk-2,+\infty).$ We get
	\begin{align*}
		g(\mu_1(G^{\prime}))\geq& g(n-b\delta+bk-2)
		\\=&b(b+1)s^2+[-2b(b+2)n-b(2b^2+5b+1)k+b(2b^2+4b+1)\delta+2b^2+5b+1]s
		\\&+3bn^2+[(8k-6\delta)b^2-(4\delta-4k+9)b-4]n+(5b^3+4b^2)k^2-b[(8b^2+8b+1)\delta
		\\&+11b+8]k+b(3b^2+4b+1)\delta^2+(9b^2+8b+1)\delta+6b+4.
	\end{align*}
	Let $g_{1}(x)=b(b+1)x^2+[-2b(b+2)n-b(2b^2+5b+1)k+b(2b^2+4b+1)\delta+2b^2+5b+1]x+3bn^2+[(8k-6\delta)b^2-(4\delta-4k+9)b-4]n+(5b^3+4b^2)k^2-b[(8b^2+8b+1)\delta+11b+8]k+b(3b^2+4b+1)\delta^2+(9b^2+8b+1)\delta+6b+4$ be a real function in $x$ with $x\in[\delta+1,\frac{n+bk-2}{b+1}].$
	It is routine to check that the derivative function of $g_{1}(x)$ is
	\begin{align*}
		g_{1}^{\prime}(x)=2b(b+1)x-2b(b+2)n-b(2b^2+5b+1)k+b(2b^2+4b+1)\delta+2b^2+5b+1.
	\end{align*}
	Hence, $\frac{2b(b+2)n+b(2b^2+5b+1)k-b(2b^2+4b+1)\delta-2b^2-5b-1}{2b(b+1)}$ is the unique solution of $g_{1}^{\prime}(x)=0$.
	Clearly,
	$$\frac{2b(b+2)n+b(2b^2+5b+1)k-b(2b^2+4b+1)\delta-2b^2-5b-1}{2b(b+1)}>\frac{n+bk-2}{b+1}.$$
	Then $g_{1}(x)$ is monotonically decreasing in $x$ with $x\in[\delta+1,\frac{n+bk-2}{b+1}]$. Therefore,
	\begin{align*}
		g_1(s)\geq&g_1(\frac{n+bk-2}{b+1})
		\\=&\frac{1}{b+1}[b^2n^2+((4b^3+5b^2+3b)k-(4b^3+6b^2+3b)\delta-3b^2-4b-3)n
		\\&+(3b^4+5b^3+3b^2)k^2-((6b^4+12b^3+8b^2+b)\delta+5b^3+8b^2+5b)k
		\\&+(3b^4+7b^3+5b^2+b)\delta^2+(5b^3+9b^2+7b+1)\delta+2b^2+4b+2].
	\end{align*}
	Let $g_2(x)=b^2x^2+[(4b^3+5b^2+3b)k-(4b^3+6b^2+3b)\delta-3b^2-4b-3]x
	+(3b^4+5b^3+3b^2)k^2-[(6b^4+12b^3+8b^2+b)\delta+5b^3+8b^2+5b]k
	+(3b^4+7b^3+5b^2+b)\delta^2+(5b^3+9b^2+7b+1)\delta+2b^2+4b+2$ be a real function in $x$ with $x\in[(2b^2+3b+11)\delta-(2b^2+\frac{5}{2}b+\frac{3}{2})k+\frac{3}{2}b+2+\frac{3}{2b},+\infty)$. We may obtain the derivative function of $g_2(x)$ is
	\begin{align*}
		g_{2}^{\prime}(x)=2b^2x+(4b^3+5b^2+3b)k-(4b^3+6b^2+3b)\delta-3b^2-4b-3.
	\end{align*}
	Hence, $\frac{(4b^3+6b^2+3b)\delta+3b^2+4b+3-(4b^3+5b^2+3b)k}{2b^2}$ is the unique solution of $g_{2}^{\prime}(x)=0$. Clearly,
	\begin{align*}
		\frac{(4b^3+6b^2+3b)\delta+3b^2+4b+3-(4b^3+5b^2+3b)k}{2b^2}<&(2b^2+3b+11)\delta-(2b^2+\frac{5}{2}b+\frac{3}{2})k
		\\&+\frac{3}{2}b+2+\frac{3}{2b}.
	\end{align*}
	Then $g_{2}(x)$ is monotonically increasing in the interval $[(2b^2+3b+11)\delta-(2b^2+\frac{5}{2}b+\frac{3}{2})k+\frac{3}{2}b+2+\frac{3}{2b}, +\infty)$. Thus,
	\begin{align*}
		g_2(n)\geq&g_2((2b^2+3b+11)\delta-(2b^2+\frac{5}{2}b+\frac{3}{2})k+\frac{3}{2}b+2+\frac{3}{2b})
		\\=&(4b^6+4b^5+32b^4+5b^3+51b^2-32b)\delta^2-[(8b^6+6b^5+27b^4-16b^3-\frac{79}{2}b^2-\frac{73}{2}b)k
		\\&-6b^5-5b^4-22b^3+\frac{23}{2}b^2+28b+\frac{73}{2}]\delta+(4b^6+2b^5-\frac{19}{4}b^4-12b^3-\frac{39}{4}b^2-\frac{9}{2}b)k^2
		\\&-(6b^5+\frac{7}{2}b^4-\frac{11}{2}b^3-\frac{39}{2}b^2-\frac{35}{2}b-9)k+\frac{9}{4}b^4+\frac{3}{2}b^3-\frac{3}{2}b^2-7b-\frac{9}{2b}-\frac{31}{4}.
	\end{align*}
	Let $g_3(x)=(4b^6+4b^5+32b^4+5b^3+51b^2-32b)x^2-[(8b^6+6b^5+27b^4-16b^3-\frac{79}{2}b^2-\frac{73}{2}b)k
	\\-6b^5-5b^4-22b^3+\frac{23}{2}b^2+28b+\frac{73}{2}]x+(4b^6+2b^5-\frac{19}{4}b^4-12b^3-\frac{39}{4}b^2-\frac{9}{2}b)k^2
	-(6b^5+\frac{7}{2}b^4-\frac{11}{2}b^3-\frac{39}{2}b^2-\frac{35}{2}b-9)k+\frac{9}{4}b^4+\frac{3}{2}b^3-\frac{3}{2}b^2-7b-\frac{9}{2b}-\frac{31}{4}$ be a real function in $x$ with $x\in[k+1, s-1]$. We may obtain the derivative function of $g_3(x)$ as
	\begin{align*}
		g_{3}^{\prime}(x)=&(8b^6+8b^5+64b^4+10b^3+102b^2-64b)x-(8b^6+6b^5+27b^4-16b^3-\frac{79}{2}b^2-\frac{73}{2}b)k
		\\&+6b^5+5b^4+22b^3-\frac{23}{2}b^2-28b-\frac{73}{2}.
	\end{align*}
	Hence, $\frac{(8b^6+6b^5+27b^4-16b^3-\frac{79}{2}b^2-\frac{73}{2}b)k-6b^5-5b^4-22b^3+\frac{23}{2}b^2+28b+\frac{73}{2}}{8b^6+8b^5+64b^4+10b^3+102b^2-64b}$ is the unique solution of $g_{3}^{\prime}(x)=0$. Clearly,
	$$\frac{(8b^6+6b^5+27b^4-16b^3-\frac{79}{2}b^2-\frac{73}{2}b)k-6b^5-5b^4-22b^3+\frac{23}{2}b^2+28b+\frac{73}{2}}{8b^6+8b^5+64b^4+10b^3+102b^2-64b}<k+1.$$
	This implies that $g_{3}(x)$ is monotonically increasing in the interval $[k+1, s-1]$. Thus,
	\begin{align*}
		g_3(\delta)\geq&g_3(k+1)
		\\=&(\frac{1}{4}b^4+9b^3+\frac{323}{4}b^2)k^2+(2b^5+\frac{77}{2}b^4+\frac{107}{2}b^3+\frac{299}{2}b^2-38b-\frac{55}{2})k
		\\&+4b^6+10b^5+\frac{157}{4}b^4+\frac{57}{2}b^3+38b^2-67b-\frac{9}{2b}-\frac{177}{4}
		\\>&0.
	\end{align*}
	Therefore, $g(\mu_1(G^{\prime}))>0$ and $g_{D(G_2)}(\mu_1(G^{\prime}))<0$. Hence, $\mu_1(G_2)>\mu_1(G^{\prime})$. Combining $\mu_1(G)\geq\mu_1(G_1)\geq\mu_1(G_2)$, we may conclude that $\mu_1(G)>\mu_1(G^{\prime})$, a contradiction to the condition.\\

	\textbf{Case 2.} $s=\delta$.

	By Lemma \ref{le:2.8.}, we have $\mu_{1}(G_1)\geq \mu_1(K_{\delta}\vee(K_{n-(b+1)\delta+bk-1}\cup (b\delta-bk+1)K_{1}))$
	with equality if and only if $G_1\cong K_{\delta}\vee(K_{n-(b+1)\delta+bk-1}\cup (b\delta-bk+1)K_{1})$. Combining $\mu_1(G)\geq\mu_1(G_1)$, we get $\mu_{1}(G)\geq \mu_1(K_{\delta}\vee(K_{n-(b+1)\delta+bk-1}\cup (b\delta-bk+1)K_{1}))$ with equality if and only if $G\cong K_{\delta}\vee(K_{n-(b+1)\delta+bk-1}\cup (b\delta-bk+1)K_{1})$. Observe that $K_{\delta}\vee(K_{n-(b+1)\delta+bk-1}\cup (b\delta-bk+1)K_{1})$ is not $k$-critical with respect to $[1,b]$-odd factors, a contradiction.\\

	\textbf{Case 3.} $s\leq\delta-1$.
	
	Recall that $G$ is a spanning subgraph of $G_1=K_{s}\vee(K_{n_{1}}\cup K_{n_{2}}\cup\cdots\cup K_{n_{bs-bk+2}})$ where $n_{1}\geq n_{2}\geq\cdots\geq n_{bs-bk+2}$ and $\sum\limits_{i=1}^{bs-bk+2}n_{i}=n-s$. Clearly, $s+n_{bs-bk+2}-1\geq\delta$ because the minimum degree of $G_1$ is at least $\delta$. Let $G_3=K_s\vee(K_{n-s-(\delta+1-s)(bs-bk+1)}\cup(bs-bk+1)K_{\delta+1-s}).$ By Lemma \ref{le:2.8.}, we have $\mu_1(G_1)\geq\mu_1(G_3)$
	with equality if and only if $(n_1,n_2,\ldots,n_{bs-bk+2} )=(n-s-(\delta+1-s)(bs-bk+1), \delta+1-s, \ldots, \delta+1-s)$.

	Consider the partition $V(G_3)=V(K_s)\cup V(K_{n-s-(\delta+1-s)(bs-bk+1)})\cup V((bs-bk+1)K_{\delta+1-s})$, the corresponding quotient matrix of $D(G_3)$ equals
	\begin{align*}
		\left(\begin{matrix}	
			s-1 & n-s-(\delta+1-s)(bs-bk+1) & (bs-bk+1)(\delta+1-s) \cr
			s &n-s-(\delta+1-s)(bs-bk+1)-1 & 2(bs-bk+1)(\delta+1-s) \cr
			s & 2(n-s-(\delta+1-s)(bs-bk+1)) & \delta-s+2(bs-bk)(\delta+1-s) \cr
		\end{matrix}\right).
	\end{align*}
	Then we obtain that the characteristic polynomial of the corresponding quotient matrix of $D(G_3)$ is given by $g_{D(G_3)}(x)$, where
	\begin{align*}
		g_{D(G_3)}(x)=&x^3+[-n+bs^2-(bk+b\delta+b)s+(b\delta+b)k+3]x^2+[2b^2s^4-(4b^2\delta+4b^2k+4b^2
		\\&-2b)s^3+(2bn+(8b^2\delta+8b^2-2b)k+2b^2k^2+2b^2\delta^2+(4b^2-7b)\delta+2b^2-5b)s^2
		\\&+(-(2bk+2b\delta+2b-3)n-(4b^2\delta+4b^2)k^2-(4b^2\delta^2+(8b^2-7b)\delta+4b^2-5b)k
		\\&+5b\delta^2+(8b-3)\delta+3b-3)s+((2b\delta+2b)k-3\delta-5)n+(2b^2\delta^2+4b^2\delta+2b^2)
		\\&\cdot k^2-(5b\delta^2+8b\delta+3b)k+3\delta^2+6\delta+6]x-b^2s^5+(2b^2k+2b^2\delta+4b^2-b)s^4
		\\&-[bn+(4b^2\delta+8b^2-b)k+b^2k^2+b^2\delta^2+(6b^2-3b)\delta+5b^2-5b]s^3+[(bk+b\delta
		\\&+3b-1)n+(2b^2\delta+4b^2)k^2+(2b^2\delta^2+(12b^2-3b)\delta+10b^2-5b)k+(2b^2-2b)
		\\&\cdot\delta^2+(4b^2-11b+1)\delta+2b^2-8b+1]s^2+[-(b^2\delta^2+6b^2\delta+5b^2)k^2-((4b^2-2b)
		\\&\cdot\delta^2+(8b^2-11b)\delta+4b^2-8b)k-((b\delta+3b)k+(2b-1)\delta+2b-4)n+(5b-1)
		\\&\cdot\delta^2+(9b-5)\delta+4b-4]s+[(2b\delta+2b)k-3\delta-4]n+(2b^2\delta^2+4b^2\delta+2b^2)k^2
		\\&-(5b\delta^2+9b\delta+4b)k+3\delta^2+6\delta+4.
	\end{align*}
	Since the partition $V(G_3)=V(K_s)\cup V(K_{n-s-(\delta+1-s)(bs-bk+1)})\cup V((bs-bk+1)K_{\delta+1-s})$ is equitable, by Lemma \ref{le:2.2.}, the largest root of $g_{D(G_3)}(x)=0$ equals $\mu_{1}(G_3)$.
	
	Let $G^{\prime}=K_\delta \vee(K_{n-(b+1)\delta+bk-1}\cup(b\delta -bk+1)K_1)$.
	Now consider the partition $V(G^{\prime})=V(K_\delta)\cup V(K_{n-(b+1)\delta+bk-1})\cup V((b\delta-bk+1)K_1)$, the corresponding quotient matrix of $D(G^{\prime})$ equals
	\begin{align*}
		\left(\begin{matrix}	
			\delta-1 & n-(b+1)\delta+bk-1 & b\delta-bk+1 \\
			\delta & n-(b+1)\delta+bk-2 & 2(b\delta-bk+1) \\
			\delta & 2(n-(b+1)\delta+bk-1) & 2(b\delta-bk) \\
		\end{matrix}\right).
	\end{align*}
	Then the characteristic polynomial of the corresponding quotient matrix of $D(G^{\prime})$ is given by $g_{D(G^{\prime})}(x)$, where
	\begin{align*}
		g_{D(G^{\prime})}(x)=&x^3+(3+bk-n-b\delta)x^2+[(2b^2+3b)\delta^2-(2bn+4b^2k+3bk-3b-3)\delta
		\\&+(2bk-5)n+2b^2k^2-3bk+6]x-(b^2+b)\delta^3+(bn+2b^2k+bk+2b^2+b-1)\delta^2
		\\&-(bkn+2bn-n+b^2k^2+4b^2k+bk-4b-2)\delta+(2bk-4)n+2b^2k^2-4bk+4.
	\end{align*}
	It is easy to see the partition
	$V(G^{\prime})=V(K_\delta)\cup V(K_{n-(b+1)\delta+bk-1})\cup V((b\delta-bk+1)K_1)$ is equitable, by Lemma \ref{le:2.2.}, the largest root of $g_{D(G^{\prime})}(x)=0$ equals $\mu_{1}(G^{\prime})$. By Lemma \ref{le:2.6.}, we get $\mu_{1}(G^{\prime})\geq\mu_{1}(K_{n-b\delta+bk-1})=n-b\delta+bk-2.$
	
	In order to prove $\mu_{1}(G_{3})>\mu_{1}(G^{\prime})$, it suffices to show $g_{D(G_3)}(\mu_{1}(G^{\prime}))<0$. Note that $g_{D(G^{\prime})}(\mu_{1}(G^{\prime}))=0$. Hence,
	\begin{align*}
		g_{D(G_3)}(\mu_1(G^{\prime}))=&g_{D(G_3)}(\mu_1(G^{\prime}))-g_{D(G^{\prime})}(\mu_1(G^{\prime}))
		\\=&(s-\delta)[ b(s-k-1)\mu_{1}^{2}(G^{\prime})+(2 b^2 s^3+
		(2 b - 4 b^2 - 4 b^2 k - 2 b^2\delta)s^2+ ( 2 b n
		\\& + 2 b^2 k^2  +
		( 4 b^2 \delta+ 8 b^2-2 b  )	k - 5 b\delta+ 2 b^2-5 b ) s+ (3 - 2 b - 2 b k) n  - (4 b^2
		\\& + 2 b^2\delta) k^2 +
		(5  - 4 b^2 + 5 b\delta)k   + ( 2 b^2+ 3 b-3)\delta + 3 b-3)\mu_{1}(G^{\prime})-b^2s^4
		\\&+(2b^2k+b^2\delta+4b^2-b)s^3-(bn+b^2k^2+(2b^2\delta+8b^2-b)k+(2b^2-2b)\delta
		\\&+5b^2-5b)s^2+((bk+3b-1)n+(b^2\delta+4b^2)k^2+((4b^2-2b)\delta+10b^2-5b)
		\\&\cdot k-(b^2+6b-1)\delta+2b^2-8b+1)s-(3bk-b\delta+2b-4)n-(5b^2+2b^2\delta)
		\\&\cdot k^2+((2b^2+6b)\delta-4b^2+8b)k-(b^2+b)\delta^2+(2b^2+b-4)\delta+4b-4].
	\end{align*}
	Let $h(x)=b(s-k-1)x^2+[2 b^2 s^3+
	(2 b - 4 b^2 - 4 b^2 k - 2 b^2\delta)s^2+ ( 2 b n
	+ 2 b^2 k^2  +
	( 4 b^2 \delta+ 8 b^2
	\\-2 b  )	k - 5 b\delta+ 2 b^2-5 b ) s+ (3 - 2 b - 2 b k) n  - (4 b^2
	+ 2 b^2\delta) k^2 +
	(5b  - 4 b^2 + 5 b\delta)k   + ( 2 b^2+ 3 b-3)\delta + 3 b
	\\-3]x-b^2s^4
	+(2b^2k+b^2\delta+4b^2-b)s^3-[bn+b^2k^2+(2b^2\delta+8b^2-b)k+(2b^2-2b)\delta
	+5b^2-5b]s^2
	\\+[(bk+3b-1)n+(b^2\delta+4b^2)k^2+((4b^2-2b)\delta+10b^2-5b)
	k-(b^2+6b-1)\delta+2b^2-8b+1]s
	\\-(3bk-b\delta+2b-4)n-(5b^2+2b^2\delta)
	k^2+[(2b^2+6b)\delta-4b^2+8b]k-(b^2+b)\delta^2+(2b^2+b-4)\delta+4b-4$
	be a real function in $x$ with $x\in[n-b\delta+bk-2,+\infty)$. Then the derivative function of $h(x)$ is
	\begin{align*}
		h^{\prime}(x)=&2b(s-k-1)x+2 b^2 s^3+
		(2 b - 4 b^2 - 4 b^2 k - 2 b^2\delta)s^2+ [ 2 b n
		+ 2 b^2 k^2
		+
		( 4 b^2 \delta+ 8 b^2
		\\&-2 b  )	k - 5 b\delta+ 2 b^2-5 b ] s+ (3 - 2 b - 2 b k) n  - (4 b^2
		+ 2 b^2\delta) k^2 +
		(5b  - 4 b^2 + 5 b\delta)k
		\\&  + ( 2 b^2+ 3 b-3)\delta + 3 b
		-3.
	\end{align*}
	Together with $k+1\leq s\leq \delta-1$ and $n\geq \frac{2}{3}b^2\delta^3+\frac{4}{3}b^2k\delta^2,$ we get
	\begin{align*}
		h^{\prime}(x)\geq& h(n-b\delta+bk-2)
		\\=&2b^2s^3-(4b^2k+2b^2\delta+4b^2-2b)s^2+[4bn+2b^2k^2+(4b^2\delta+10b^2-2b)k
		-(5b+2b^2)\delta
		\\&+2b^2-9b]s
		-(4bk+4b-3)n-(6b^2+2b^2\delta)k^2+[(2b^2+5b)\delta-6b^2+9b]k+(4b^2
		\\&+3b-3)\delta+7b-3
		\\\geq&2b^2(k+1)^3-(4b^2k+2b^2\delta+4b^2-2b)(\delta-1)^2+[4bn+2b^2k^2+(4b^2\delta+10b^2-2b)k
		\\&-(5b+2b^2)\delta+2b^2-9b](k+1)
		-(4bk+4b-3)n-(6b^2+2b^2\delta)k^2+[(2b^2+5b)\delta
		\\&-6b^2+9b]k+(4b^2+3b-3)\delta+7b-3
		\\=&3n-2b^2\delta^3-(4b^2k-2b)\delta^2+(2b^2k^2+12b^2k+8b^2-6b-3)\delta+4b^2k^3+(12b^2-2b)k^2
		\\&+(8b^2-2b)k-3
		\\\geq&3(\frac{2}{3}b^2\delta^3+\frac{4}{3}b^2k\delta^2)-2b^2\delta^3-(4b^2k-2b)\delta^2+(2b^2k^2+12b^2k+8b^2-6b-3)\delta+4b^2k^3
		\\&+(12b^2-2b)k^2+(8b^2-2b)k-3
		\\=&2b^2\delta+(2b^2k^2+12b^2k+8b^2-6b-3)\delta+4b^2k^3+(12b^2-2b)k^2+(8b^2-2b)k-3
		\\>&0.
	\end{align*}
	This implies that $h(x)$ is monotonically increasing in the interval  $[n-b\delta+bk-2,+\infty)$. Therefore,
	\begin{align*}
		h(\mu_{1}(G^{\prime}))\geq&h(n-b\delta+bk-2)
		\\=&-b^2s^4+[2b^2n+(2b^3+2b^2)k+(b^2-2b^3)\delta-b]s^3+[-(4b^2k+2b^2\delta+4b^2-b)n
		\\&-(4b^3+b^2)k^2+((2b^3-2b^2)\delta-4b^3+2b^2+b)k+2b^3\delta^2+(4b^3+2b)\delta+3b^2+b]
		\\&\cdot s^2+[3bn^2+(2b^2k^2+(4b^2\delta+12b^2-b)k-(4b^2+5b)\delta+2b^2-10b-1)n
		\\&+2b^3k^3+((2b^3+b^2)\delta+9b^3-2b^2)k^2-(4b^3\delta^2+(10b^3+7b^2+2b)\delta-2b^3+15b^2
		\\&+b)k+(b^3+5b^2)\delta^2-(2b^3-8b^2-4b-1)\delta-2b^2+6b+1]s-(3bk+3b-3)n^2
		\\&-[(2b^2\delta+8b^2)k^2-((4b^2+5b)\delta-8b^2+13b)k-(6b^2+b-3)\delta-9b+5]n
		\\&-(2b^3\delta+5b^3)k^3+[2b^3\delta^2+(7b^2+6b^3)\delta-5b^3+12b^2]k^2-[(b^3+5b^2)\delta^2-(8b^3
		\\&-4b^2-7b)\delta-11b^2+9b]k-(3b^3+4b^2-2b)\delta^2-(9b^2+2b-2)\delta-6b+2
		\\\geq&-b^2(\delta-1)^4+[2b^2n+(2b^3+2b^2)k+(b^2-2b^3)\delta-b](k+1)^3+[-(4b^2k+2b^2\delta
		\\&+4b^2-b)n-(4b^3+b^2)k^2+((2b^3-2b^2)\delta-4b^3+2b^2+b)k+2b^3\delta^2+(4b^3
		\\&+2b)\delta+3b^2+b] (\delta-1)^2+[3bn^2+(2b^2k^2+(4b^2\delta+12b^2-b)k-(4b^2+5b)\delta
		\\&+2b^2-10b-1)n+2b^3k^3+((2b^3+b^2)\delta+9b^3-2b^2)k^2-(4b^3\delta^2+(10b^3+7b^2
		\\&+2b)\delta-2b^3+15b^2+b)k+(b^3+5b^2)\delta^2-(2b^3-8b^2-4b-1)\delta-2b^2+6b+1]
		\\&\cdot(k+1)-(3bk+3b-3)n^2-[(2b^2\delta+8b^2)k^2-((4b^2+5b)\delta-8b^2+13b)k
		\\&-(6b^2+b-3)\delta-9b+5]n-(2b^3\delta+5b^3)k^3+[2b^3\delta^2+(7b^2+6b^3)\delta-5b^3
		\\&+12b^2]k^2-[(b^3+5b^2)\delta^2-(8b^3-4b^2-7b)\delta-11b^2+9b]k-(3b^3+4b^2-2b)\delta^2
		\\&-(9b^2+2b-2)\delta-6b+2
		\\=&3n^2+[4b^2k^3+(2b^2\delta+12b^2-b)k^2-(4b^2\delta^2-12b^2\delta-8b^2-2b+1)k-2b^2\delta^3
		\\&+b\delta^2+(8b^2-6b-3)\delta-6]n+(4b^3+2b^2)k^4+[(2b^2-2b^3)\delta+12b^3+4b^2-b]
		\\&\cdot k^3-[(6b^3+b^2)\delta^2-(6b^2-2b)\delta-8b^3+4b]k^2+[(2b^3-2b^2)\delta^3-(12b^3-6b^2
		\\&-b)\delta^2-(6b^2+7b-1)\delta-2b^2-6b+1]k+(2b^3-b^2)\delta^4+(4b^2+2b)\delta^3-(8b^3
		\\&+2b^2+b)\delta^2-(2b^2-2b-3)\delta+3.
	\end{align*}
	Let $h_{1}(x)=3x^2+[4b^2k^3+(2b^2\delta+12b^2-b)k^2-(4b^2\delta^2-12b^2\delta-8b^2-2b+1)k-2b^2\delta^3+b\delta^2+(8b^2-6b-3)\delta-6]x+(4b^3+2b^2)k^4+[(2b^2-2b^3)\delta+12b^3+4b^2-b]k^3-[(6b^3+b^2)\delta^2-(6b^2-2b)\delta-8b^3+4b]k^2+[(2b^3-2b^2)\delta^3-(12b^3-6b^2-b)\delta^2-(6b^2+7b-1)\delta-2b^2-6b+1]k+(2b^3-b^2)\delta^4+(4b^2+2b)\delta^3-(8b^3+2b^2+b)\delta^2-(2b^2-2b-3)\delta+3$ be a real function in $x$ with $x\in[\frac{2}{3}b^2\delta^3+\frac{4}{3}b^2k\delta^2,+\infty).$
	We may obtain the derivative function of $h_{1}(x)$ is
	\begin{align*}
		h_{1}^{\prime}(x)=&6x+4b^2k^3+(2b^2\delta+12b^2-b)k^2-(4b^2\delta^2-12b^2\delta-8b^2-2b+1)k-2b^2\delta^3+b\delta^2
		\\&+(8b^2-6b-3)\delta-6.
	\end{align*}
	Hence, $\frac{-4b^2k^3-(2b^2\delta+12b^2-b)k^2+(4b^2\delta^2-12b^2\delta-8b^2-2b+1)k+2b^2\delta^3-b\delta^2-(8b^2-6b-3)\delta+6}{6}$ is the unique solution of $h_{1}^{\prime}(x)=0$.
	Clearly,
	\begin{align*}
		&	6(\frac{2}{3}b^2\delta^3+\frac{4}{3}b^2k\delta^2)+4b^2k^3+(2b^2\delta+12b^2-b)k^2-(4b^2\delta^2-12b^2\delta-8b^2-2b+1)k-2b^2\delta^3
		\\&+b\delta^2+(8b^2-6b-3)\delta+6>0.
	\end{align*}
	This implies that $h_{1}(x)$ is monotonically increasing in the interval $[\frac{2}{3}b^2\delta^3+\frac{4}{3}b^2k\delta^2,+\infty)$. Therefore,
	\begin{align*}
		h_1(n)\geq&h_1(\frac{2}{3}b^2\delta^3+\frac{4}{3}b^2k\delta^2)
		\\=&\frac{2}{3}b^3\delta^5+[\frac{4}{3}b^4k^2+(\frac{4}{3}b^3+8b^4)k+\frac{16}{3}b^4-2b^3-3b^2]\delta^4+[\frac{16}{3}b^4k^4+(24b^4-\frac{2}{3}b^3)k^2
		\\&+(16b^4-\frac{14}{3}b^3-\frac{20}{3}b^2)k+2b]\delta^3+[\frac{16}{3}b^4k^4+(16b^4-\frac{4}{3}b^3)k^3+(\frac{32}{3}b^4-\frac{10}{3}b^3-\frac{7}{3}b^2)
		\\&\cdot k^2-(12b^3+2b^2-b)k-8b^3-2b^2-b]\delta^2+[(2b^2-2b^3)k^3+(6b^2-2b)k^2+(1-7b
		\\&-6b^2)k-2b^2+2b+3]\delta+(4b^3+2b^3)k^4+(12b^3+4b^2-b)k^3+(8b^3-4b)k^2-(2b^2
		\\&+6b-1)k+3
		\\>&0.
	\end{align*}
	
	Therefore, $h(\mu_1(G^{\prime}))>0$ and $g_{D(G_3)}(\mu_1(G^{\prime}))<0$. Hence, $\mu_1(G_3)>\mu_1(G^{\prime})$. Combining $\mu_1(G)\geq\mu_1(G_1)\geq\mu_1(G_3)$, we may conclude that $\mu_1(G)>\mu_1(G^{\prime})$, a contradiction to the condition.	
	
	This completes the proof.
	$\hfill\square$\\
	
	\section{Proof of Theorem \ref{th:1.6.} }
	\hspace{1.3em}
	In this section, we will give the proof of Theorem 1.6.\\
	
	\noindent\textbf{Proof of Theorem \ref{th:1.6.}.}
	
	Suppose that $G$ is a $(k+1)$-connected graph that is not $k$-critical with respect to $[1,b]$-odd factors. Then by Lemma \ref{le:2.1.} (with $f=b$), there exists a subset $S\subseteq V(G)$ such that $|S|=s\geq k$ satisfying $o(G-S)\geq bs-bk+1$. Since $k\equiv n\equiv s+o(G-S)\pmod{2}$ and $b\equiv 1\pmod{2}$, it follows that $s-k\equiv o(G-S)\pmod{2}$ and $s-k\equiv b(s-k)\pmod{2}.$ Thus, $o(G-S)\geq b(s-k)+2$. This implies $$n\geq s+o(G-S)\geq s+b(s-k)+2\Longrightarrow s\leq\frac{n+bk-2}{b+1}.$$ Then $G$ is a spanning subgraph of $G_1=K_{s}\vee(K_{n_{1}}\cup K_{n_{2}}\cup\cdots\cup K_{n_{bs-bk+2}})$ for some positive odd integers $n_{1}\geq n_{2}\geq\cdots\geq n_{bs-bk+2}$ and $\sum\limits_{i=1}^{bs-bk+2}n_{i}=n-s$.
	By Lemma \ref{le:2.5.}, we obtain
	\begin{align*}
		\eta_{1}(G)\geq\eta_1(G_1)
	\end{align*}
	with equality if and only if $G\cong G_1$.
	If $s=k$, then $o(G-S)\geq 2$, contradicting the $(k+1)$-connectivity of $G$. Thus, $s\geq k+1$. We proceed by considering the following three possible cases.\\

	\textbf{Case 1.} $s\geq\delta+1$.
	
	Let $G_2=K_s\vee(K_{n-(b+1)s+bk-1}\cup(bs-bk+1)K_1)$. By Lemma \ref{le:2.9.}, we have
	\begin{align*}
		\eta_1(G_1)\geq\eta_1(G_2)
	\end{align*}
	with equality if and only if $(n_1,n_2,\ldots,n_{bs-bk+2})=(n-(b+1)s+bk-1,1,\ldots,1)$.

	Now consider the partition $V(G_2)=V(K_s)\cup V(K_{n-(b+1)s+bk-1})\cup V((bs-bk+1)K_1)$, the corresponding quotient matrix of $Q_{D}(G_2)$ is
	\begin{align*}
		\left(\begin{matrix}	
			n+s-2 & n-(b+1)s+bk-1 & bs-bk+1 \cr
			s & 2n-s-2 & 2(bs-bk+1) \cr
			s & 2(n-(b+1)s+bk-1) & 2(n-bk-1)+(2b-1)s \cr
		\end{matrix}\right).
	\end{align*}
	Then we obtain that the characteristic polynomial of the corresponding quotient matrix of $Q_{D}(G_2)$ is given by $g_{Q_{D}(G_2)}(x)$, where
	\begin{align*}
		g_{Q_{D}(G_2)}(x)=&x^3+[(1-2b)s-5n+2bk+6]x^2+[(4b^2+4b)s^2+((2b-3)n-(8b^2+4b)k
		\\&+8)s+8n^{2}-(2bk+24)n+4b^2k^2+16]x-2b^2s^3+[-(4b^2+4b)n+4b^2k+8b^2
		\\&+6b]s^2+[2n^2+((8b^2+4b)k-4b-10)n-2b^2k^2-(16b^2+6b)k+8b+12]s
		\\&-4n^3+20n^2-(4b^2k^2-4bk+32)n+8b^2k^2-8bk+16.
	\end{align*}
	Since the partition $V(G_2)=V(K_s)\cup V(K_{n-(b+1)s+bk-1})\cup V((bs-bk+1)K_1)$ is equitable, by Lemma \ref{le:2.2.}, the largest root of $g_{Q_{D}(G_2)}(x)=0$ equals $\eta_{1}(G_2)$.
	
	Let $G^{\prime}=K_\delta \vee(K_{n-(b+1)\delta+bk-1}\cup(b\delta -bk+1)K_1)$.
	Now consider the partition $V(G^{\prime})=V(K_\delta)\cup V(K_{n-(b+1)\delta+bk-1})\cup V((b\delta-bk+1)K_1)$, the corresponding quotient matrix of $Q_{D}(G^{\prime})$ is
	\begin{align*}
		\left(\begin{matrix}	
			n+\delta-2 & n-(b+1)\delta+bk-1 & b\delta-bk+1 \cr
			\delta & 2n-\delta-2 & 2(b\delta-bk+1) \cr
			\delta & 2(n-(b+1)\delta+bk-1) & 2(n-bk-1)+(2b-1)\delta \cr
		\end{matrix}\right).
	\end{align*}
	Then the characteristic polynomial of the corresponding quotient matrix of $Q_{D}(G^{\prime})$ is given by $g_{Q_{D}(G^{\prime})}(x),$ where
	\begin{align*}
		g_{Q_{D}(G^{\prime})}(x)=&x^3+[(1-2b)\delta-5n+2bk+6]x^2+[(4b^2+4b)\delta^2+((2b-3)n-(8b^2+4b)k
		\\&+8)\delta+8n^{2}-(2bk+24)n+4b^2k^2+16]x-2b^2\delta^3+[-(4b^2+4b)n+4b^2k+8b^2
		\\&+6b]\delta^2+[2n^2+((8b^2+4b)k-4b-10)n-2b^2k^2-(16b^2+6b)k+8b+12]\delta
		\\&-4n^3+20n^2-(4b^2k^2-4bk+32)n+8b^2k^2-8bk+16.
	\end{align*}
	Observe that the partition $V(G^{\prime})=V(K_\delta)\cup V(K_{n-(b+1)\delta+bk-1})\cup V((b\delta-bk+1)K_1)$ is equitable, by Lemma \ref{le:2.2.}, the largest root of $g_{Q_{D}(G^{\prime})}(x)=0$ equals $\eta_{1}(G^{\prime})$.
	
	According to the definition of the Wiener index, we have
	\begin{align*}
		W(G^{\prime})=&\sum\limits_{i<j}d_{ij}({G}^{\prime})
		\\=&\frac{1}{2}\delta(\delta-1)+[n-(b+1)\delta+bk-1]\delta+(b\delta-bk+1)\delta+(b\delta-bk+1)(b\delta-bk)
		\\&+\frac{1}{2}[n-(b+1)\delta+bk-1][n-(b+1)\delta+bk-2]
		\\&+2[n-(b+1)\delta+bk-1](b\delta-bk+1)
		\\=&\frac{1}{2}[n^2+(2b\delta-2bk+1)n-(b^2+2b)\delta^2+((2b^2+2b)k-3b-2)\delta-b^2k^2+3bk-2].
	\end{align*}
	By Lemma 2.7, since $n\geq2(b^2+2b)\delta^2+2\delta+2b^2k^2$, we get
	\begin{align*}
		\eta_{1}(G^{\prime})\geq&\frac{4W(G^{\prime})}{n}	
		\\=&2[n+2b\delta-2bk+1-\frac{(b^2+2b)\delta^2-((2b^2+2b)k-3b-2)\delta+b^2k^2-3bk+2}{n}]
		\\=&2n+4b\delta-4bk+1+\frac{n-2b(b+2)\delta^2+2[b(2b+2)k-3b-2]\delta-2b^2k^2+6bk-4}{n}
		\\\geq&2n+4b\delta-4bk+1+\frac{2[2kb^2+(2k-3)b-1]\delta+6bk-4}{n}
		\\\geq&2n+4b\delta-4bk+1+\frac{2(2k^3+2k^2-3k-1)\delta+6k^2-4}{n}
		\\>&2n+4b\delta-4bk+1.
	\end{align*}

	In order to prove $\eta_{1}(G_{2})>\eta_{1}(G^{\prime})$, it suffices to show $g_{Q_{D}(G_2)}(\eta_{1}(G^{\prime}))<0$. Note that $g_{Q_{D}(G^{\prime})}(\eta_{1}(G^{\prime}))=0$. Hence,
	\begin{align*}
		g_{Q_{D}(G_2)}(\eta_1(G^{\prime}))=&g_{Q_{D}(G_2)}(\eta_1(G^{\prime}))-g_{Q_{D}(G^{\prime})}(\eta_1(G^{\prime}))
		\\=&(\delta-s)[ (2b-1)\eta_{1}^{2}(G^{\prime})-((2b-3)n+(4b^2+4b)(s+\delta)-(4b+8b^2)k+8)
		\\&\cdot \eta_{1}(G^{\prime})+2b^2s^2+((4b^2+4b)n+2b^2\delta-4b^2k-8b^2-6b)s-2n^2
		\\&+((4b^2+4b)\delta-(4b+8b^2)k+4b+10)n+2b^2k^2-(4b^2\delta-16b^2-6b)k
		\\&+2b^2\delta^2-(6b+8b^2)\delta-8b-12].
	\end{align*}
	Let $w(x)=(2b-1)x^{2}-[(2b-3)n+(4b^2+4b)(s+\delta)-(4b+8b^2)k+8] x+2b^2s^2+[(4b^2+4b)n+2b^2\delta-4b^2k-8b^2-6b]s-2n^2+[(4b^2+4b)\delta-(4b+8b^2)k+4b+10]n+2b^2k^2-(4b^2\delta-16b^2-6b)k+2b^2\delta^2-(6b+8b^2)\delta-8b-12$
	be a real function in $x$ with $x\in[2n+4b\delta-4bk+1,+\infty)$. Then the axis of symmetry of $w(x)$ is
	\begin{align*}
		x=\frac{(2b-3)n+(4b^2+4b)(s+\delta)-(4b+8b^2)k+8}{2(2b-1)}.
	\end{align*}
	Since $\delta\geq k+1$ and $n\geq(b+1)s-bk+2$, we have
	\begin{align*}
		&2(2b-1)(2n+4b\delta-4bk+1)-[(2b-3)n+(4b^2+4b)(s+\delta)-(4b+8b^2)k+8]
		\\=&(6b-1)n+(12b^2-12b)\delta-(8b^2-12b)k-(4b^2+4b)s+4b-10
		\\\geq&(2b-1)n+(12b^2-12b)\delta-(8b^2-12b)k-(4b^2+4b)s+4b-10+4b[(b+1)s-bk+2]
		\\=&(2b-1)n+12b(b-1)\delta-12b(b-1)k+12b-10
		\\\geq&(2b-1)n+12b^2-10
		\\>&0.
	\end{align*}
	Consequently, $w(x)$ is monotonically increasing in the interval $[2n+4b\delta-4bk+1,+\infty).$
	We have
	\begin{align*}
		w(\eta_1(G^{\prime}))\geq& w(2n+4b\delta-4bk+1)
		\\=&2b^2s^2-[(4b+4b^2)n+(14b^2+16b^3)\delta-(12b^2+16b^3)k+12b^2+10b]s
		\\&+4bn^2+[(20b^2-8b)\delta-(16b^2-8b)k+10b-7]n-30b^2k^2+(8b^2+50b)k
		\\&+(16b^3-30b^2)\delta^2+[(60b^2-16b^3)k+4b^2-50b]\delta-6b-21.
	\end{align*}
	Let $w_{1}(x)=2b^2x^2-[(4b+4b^2)n+(14b^2+16b^3)\delta-(12b^2+16b^3)k+12b^2+10b]x+4bn^2+[(20b^2-8b)\delta-(16b^2-8b)k+10b-7]n-30b^2k^2+(8b^2+50b)k+(16b^3-30b^2)\delta^2+[(60b^2-16b^3)k+4b^2-50b]\delta-6b-21$ be a real function in $x$ with $x\in[\delta+1,\frac{n+bk-2}{b+1}].$
	It is routine to check that the axis of symmetry of $w_{1}(x)$ is
	\begin{align*}
		x=\frac{(4b+4b^2)n+(14b^2+16b^3)\delta-(12b^2+16b^3)k+12b^2+10b}{4b^2}.
	\end{align*}
	By some calculations, we get
	$$\frac{(4b+4b^2)n+(14b^2+16b^3)\delta-(12b^2+16b^3)k+12b^2+10b}{4b^2}>\frac{n+bk-2}{b+1}.$$
	Then $w_{1}(x)$ is monotonically decreasing in $x$ with $x\in[\delta+1,\frac{n+bk-2}{b+1}]$. Therefore,
	\begin{align*}
		w_1(s)\geq&w_1(\frac{n+bk-2}{b+1})
		\\=&\frac{1}{(b+1)^{2}}[2b^2n^2+((4b^4+2b^3-10b^2-8b)\delta+(-4b^4+8b^2+8b)k+6b^3-b^2-6b-7)
		\\&\cdot n+(16b^5+2b^4-44b^3-30b^2)\delta^2+((-32b^5-2b^4+90b^3+60b^2)k+36b^4+18b^3
		\\&-68b^2-50b)\delta+(16b^5-48b^3-30b^2)k^2+(-36b^4-20b^3+74b^2+50b)k+18b^3
		\\&+19b^2-28b-21].
	\end{align*}
	Let $w_2(x)=2b^2x^2+[(4b^4+2b^3-10b^2-8b)\delta+(-4b^4+8b^2+8b)k+6b^3-b^2-6b-7]x+(16b^5+2b^4-44b^3-30b^2)\delta^2+[(-32b^5-2b^4+90b^3+60b^2)k+36b^4+18b^3-68b^2-50b]\delta+(16b^5-48b^3-30b^2)k^2+(-36b^4-20b^3+74b^2+50b)k+18b^3+19b^2-28b-21$ be a real function in $x$ with $x\in[2(b^2+2b)\delta^2+2\delta+2b^2k^2,+\infty)$. We may obtain the derivative function of $w_2(x)$ is
	\begin{align*}
		w_{2}^{\prime}(x)=4b^2x+(4b^4+2b^3-10b^2-8b)\delta+(-4b^4+8b^2+8b)k+6b^3-b^2-6b-7
	\end{align*}
	and
	\begin{align*}
		w_{2}^{\prime}(x)\geq&4b^2[2(b^2+2b)\delta^2+2\delta+2b^2k^2]+(4b^4+2b^3-10b^2-8b)\delta+(-4b^4+8b^2+8b)k\\&+6b^3-b^2-6b-7
		\\=&(8b^4+16b^3)\delta^2+(4b^4+2b^3-2b^2-8b)\delta+8b^4k^2+(-4b^4+8b^2+8b)k
		\\&+6b^3-b^2-6b-7
		\\\geq&(8b^4+16b^3)(k+1)^2+(4b^4+2b^3-2b^2-8b)(k+1)+8b^4k^2+(-4b^4+8b^2+8b)k
		\\&+6b^3-b^2-6b-7
		\\=&16(b^4+b^3)k^2+(16b^4+34b^3+6b^2)k+12b^4+24b^3-3b^2-14b-7
		\\>&0.
	\end{align*}
	Clearly, $w_{2}(x)$ is montonically increasing in the interval $[2(b^2+2b)\delta^2+2\delta+2b^2k^2,+\infty)$. Thus,
	\begin{align*}
		w_2(n)\geq&w_2(2(b^2+2b)\delta^2+2\delta+2b^2k^2)
		\\=&(8b^6+32b^5+32b^4)\delta^4+(8b^6+20b^5+4b^4-24b^3-32b^2)\delta^3+[(16b^6+32b^5)k^2
		\\&+(-8b^6-16b^5+16b^4+48b^3+32b^2)k+28b^5+32b^4-56b^3-80b^2-44b]\delta^2
		\\&[(8b^6+4b^5-4b^4-16b^3)k^2+(-32b^5-10b^4+90b^3+76b^2+16b)k+36b^4+30b^3
		\\&-70b^2-62b-14]\delta+8b^6k^4+(-8b^6+16b^4+16b^3)k^3+(28b^5-2b^4-60b^3-44b^2)
		\\&\cdot k^2+(-36b^4-20b^3+74b^2+50b)k+18b^3+19b^2-28b-21
		\\\geq&(8b^6+32b^5+32b^4)(k+1)^4+(8b^6+20b^5+4b^4-24b^3-32b^2)(k+1)^3
		\\&+[(16b^6+32b^5)k^2+(-8b^6-16b^5+16b^4+48b^3+32b^2)k+28b^5+32b^4-56b^3
		\\&-80b^2-44b](k+1)^2+[(8b^6+4b^5-4b^4-16b^3)k^2+(-32b^5-10b^4+90b^3+76b^2
		\\&+16b)k+36b^4+30b^3-70b^2-62b-14](k+1)+8b^6k^4+(-8b^6+16b^4+16b^3)k^3
		\\&+(28b^5-2b^4-60b^3-44b^2)k^2+(-36b^4-20b^3+74b^2+50b)k+18b^3+19b^2-28b
		\\&-21
		\\=&(32k^4+64k^3+80k^2+48k+16)b^6+(64k^4+200k^3+280k^2+196k+80)b^5
		\\&+(32k^4+160k^3+252k^2+210k+104)b^4+(24k^3-18k^2-36k-32)b^3
		\\&-(80k^2+144k+163)b^2-(28k^2+84k+134)b-14k-35
		\\\geq&32k^{10}+128k^9+312k^8+488k^7+488k^6+272k^5-12k^4-204k^3-247k^2-148k
		\\&-35
		\\>&0.
	\end{align*}
	
	Therefore, we can conclude that $w(\eta_1(G^{\prime}))>0$ and $g_{Q_{D}(G_2)}(\eta_1(G^{\prime}))<0$. Hence, $\eta_1(G_2)>\eta_{1}(G^{\prime})$. Combining $\eta_1(G)\geq\eta_1(G_1)\geq\eta_1(G_2)$, we may conclude that $\eta_1(G)>\eta_1(G^{\prime})$, a contradiction to the condition.\\

	\textbf{Case 2.} $s=\delta$.

	By Lemma \ref{le:2.9.}, we have $\eta_{1}(G_1)\geq \eta_1(K_{\delta}\vee(K_{n-(b+1)\delta+bk-1}\cup (b\delta-bk+1)K_{1}))$
	with equality if and only if $G_1\cong K_{\delta}\vee(K_{n-(b+1)\delta+bk-1}\cup (b\delta-bk+1)K_{1})$. Combining $\eta_1(G)\geq\eta_1(G_1)$, we get $\eta_{1}(G)\geq \eta_1(K_{\delta}\vee(K_{n-(b+1)\delta+bk-1}\cup (b\delta-bk+1)K_{1}))$ with equality if and only if $G\cong K_{\delta}\vee(K_{n-(b+1)\delta+bk-1}\cup (b\delta-bk+1)K_{1})$. Observe that $K_{\delta}\vee(K_{n-(b+1)\delta+bk-1}\cup (b\delta-bk+1)K_{1})$ is not $k$-critical with respect to $[1,b]$-odd factors, a contradiction.\\

	\textbf{Case 3.} $s\leq\delta-1$.
	
	Recall that $G$ is a spanning subgraph of $G_1=K_{s}\vee(K_{n_{1}}\cup K_{n_{2}}\cup\cdots\cup K_{n_{bs-bk+2}})$ where $n_{1}\geq n_{2}\geq\cdots\geq n_{bs-bk+2}$ and $\sum\limits_{i=1}^{bs-bk+2}n_{i}=n-s$. Clearly, $s+n_{bs-bk+2}-1\geq\delta$ because the minimum degree of $G_1$ is at least $\delta$. Combining with $s\geq k+1$ and $b\geq k$, we have $(bs-bk+2)-s-1=(b-1)s-bk+1\geq(b-1)(k+1)-bk+1=b-k\geq0.$ We assert that $n_{1}\geq5(\delta-s+1).$ Suppose to the contrary that $n_{1}\leq(5\delta-5s+4).$ Notice that $n_{1}\geq n_{2}\geq\cdots\geq n_{bs-bk+2}\geq \delta+1-s$, we have
	\begin{align*}
		n=&s+n_{1}+ n_{2}+\cdots+ n_{bs-bk+2}\\
		\leq& s+(5\delta-5s+4)(bs-bk+2)\\
		=&-5bs^2+(5b\delta+5bk+4b-9)s-5bk\delta+10\delta-4bk+8
		\\\leq&-5b(\frac{5b\delta+5bk+4b-9}{10b})^2+\frac{(5b\delta+5bk+4b-9)^{2}}{10b}-5bk\delta+10\delta-4bk+8
		\\=&\frac{5}{4}b\delta^2+(-\frac{5}{2}bk+2b+\frac{11}{2})\delta+\frac{5}{4}bk^2-2bk-\frac{9}{2}k+\frac{4}{5}b+\frac{81}{20b}+\frac{22}{5}
	\end{align*}
	and
	\begin{align*}
		&2(b^2+2b)\delta^2+2\delta+2b^2k^2-[\frac{5}{4}b\delta^2+(-\frac{5}{2}bk+2b+\frac{11}{2})\delta+\frac{5}{4}bk^2-2bk-\frac{9}{2}k+\frac{4}{5}b+\frac{81}{20b}
		\\&+\frac{22}{5}]\\
		=&(2b^2+\frac{11}{4}b)\delta^2+(\frac{5}{2}bk-2b-\frac{7}{2})\delta+2b^2k^2-\frac{5}{4}bk^2+2bk+\frac{9}{2}k-\frac{4}{5}b-\frac{81}{20b}-\frac{22}{5} \\
		\geq&(4k^2+4k+2)b^2+(4k^2+8k-\frac{1}{20})b+k-\frac{81}{20b}-\frac{79}{10}
		\\\geq&4k^4+8k^3+10k^2+\frac{19}{20}k-\frac{81}{20k}-\frac{79}{10}
		\\>&0.
	\end{align*}
	This contradicts to the condition of $n$.
	
	Let $G_3=K_s\vee(K_{n-s-(\delta+1-s)(bs-bk+1)}\cup(bs-bk+1)K_{\delta+1-s}).$ By Lemma \ref{le:2.10.}, we have $\eta_1(G_1)\geq\eta_1(G_3)$
	with equality if and only if $(n_1,n_2,\ldots,n_{bs-bk+2} )=(n-s-(\delta+1-s)(bs-bk+1),\delta+1-s,\ldots,\delta+1-s)$.

	Consider the partition $V(G_3)=V(K_s)\cup V(K_{n-s-(\delta+1-s)(bs-bk+1)})\cup V((bs-bk+1)K_{\delta+1-s})$, the corresponding quotient matrix of $Q_{D}(G_3)$ equals
	\begin{align*}
		\left(\begin{matrix}	
			n+s-2 & n-s-(\delta+1-s)(bs-bk+1) & (bs-bk+1)(\delta+1-s) \cr
			s &2n-s-2 & 2(bs-bk+1)(\delta+1-s) \cr
			s & 2(n-s-(\delta+1-s)(bs-bk+1)) & 2(n-1+(bs-bk)(\delta+1-s))-s \cr
		\end{matrix}\right).
	\end{align*}
	Then we obtain that the characteristic polynomial of the corresponding quotient matrix of $Q_{D}(G_3)$ is given by $g_{Q_{D}(G_3)}(x)$, where
	\begin{align*}
		g_{Q_{D}(G_3)}(x)=&x^3+[2bs^2-(2b\delta+2bk+2b-1)s-5n+2bk\delta+2bk+6]x^2+[4b^2s^4+(4b
		\\&-8b^2-8b^2k-8b^2\delta)s^3+(4b^2\delta^2+(16b^2k+8b^2-12b)\delta-2bn+4b^2k^2+(16b^2
		\\&-4b)k+4b^2-4b)s^2+((2bk+2b\delta+2b+1)n-8b^2k^2+(4b-8b^2)k+(8b
		\\&-8b^2k)\delta^2+(-8b^2k^2+(-16b^2+12b)k+8b-4)\delta)s+8n^2-((2bk+4)\delta
		\\&+2bk+24)n+4b^2k^2+(4b^2k^2-8bk+4)\delta^2+(8b^2k^2-8bk+8)\delta+16]x
		\\&- 2b^2s^5+(-4b^2n+4b^2\delta+4b^2k+12b^2-2b)s^4+[(8b^2\delta+8b^2k+8b^2-4b)n
		\\&-2b^2\delta^2-(8b^2k+20b^2-4b)\delta-2b^2k^2+(-24b^2+2b)k-18b^2+12b]s^3
		\\&+[(-4b^2\delta^2+(-16b^2k-8b^2+12b)\delta-4b^2k^2+(-16b^2+4b)k-4b^2+8b)n
		\\&+(4b^2k+8b^2-2b)\delta^2+(4b^2k^2+(40b^2-4b)k+16b^2-28b)\delta+12b^2k^2+(36b^2
		\\&-12b)k+8b^2-18b]s^2+[-2n^2+((8b^2k-8b)\delta^2+(8b^2k^2+(16b^2-12b)k
		\\&-12b+4)\delta+8b^2k^2+(8b^2-8b)k-4b+6)n+(-2b^2k^2+(-16b^2+2b)k
		\\&+16b)\delta^2+(-20b^2k^2+(-32b^2+28b)k+24b-8)\delta-18b^2k^2+(-16b^2+18b)k
		\\&+8b-4]s-4n^3+(4\delta+20)n^2+[(-4b^2k^2+8bk-4)\delta^2+(-8b^2k^2+12bk
		\\&-16)\delta-4b^2k^2+4bk-32]n+(8b^2k^2-16bk+8)\delta^2+(16b^2k^2-24bk+16)\delta
		\\&+8b^2k^2-8bk+16.
	\end{align*}
	Since the partition $V(G_3)=V(K_s)\cup V(K_{n-s-(\delta+1-s)(bs-bk+1)})\cup V((bs-bk+1)K_{\delta+1-s})$ is equitable, by Lemma \ref{le:2.2.}, the largest root of $g_{Q_{D}(G_3)}(x)=0$ equals $\eta_{1}(G_3)$.
	
	Let $G^{\prime}=K_\delta \vee(K_{n-(b+1)\delta+bk-1}\cup(b\delta -bk+1)K_1)$.
	Now consider the partition $V(G^{\prime})=V(K_\delta)\cup V(K_{n-(b+1)\delta+bk-1})\cup V((b\delta-bk+1)K_1)$, the corresponding quotient matrix of $Q_{D}(G^{\prime})$ equals
	\begin{align*}
		\left(\begin{matrix}	
			n+\delta-2 & n-(b+1)\delta+bk-1 & b\delta-bk+1 \cr
			\delta & 2n-\delta-2 & 2(b\delta-bk+1) \cr
			\delta & 2(n-(b+1)\delta+bk-1) & 2(n-bk-1)+(2b-1)\delta \cr
		\end{matrix}\right).
	\end{align*}
	Then the characteristic polynomial of the corresponding quotient matrix of $Q_{D}(G^{\prime})$ is given by $g_{Q_{D}(G^{\prime})}(x),$ where
	\begin{align*}
		g_{Q_{D}(G^{\prime})}(x)=&x^3+[(1-2b)\delta-5n+2bk+6]x^2+[(4b^2+4b)\delta^2+((2b-3)n-(8b^2+4b)k
		\\&+8)\delta+8n^{2}-(2bk+24)n+4b^2k^2+16]x-2b^2\delta^3+[-(4b^2+4b)n+4b^2k+8b^2
		\\&+6b]\delta^2+[2n^2+((8b^2+4b)k-4b-10)n-2b^2k^2-(16b^2+6b)k+8b+12]\delta
		\\&-4n^3+20n^2-(4b^2k^2-4bk+32)n+8b^2k^2-8bk+16.
	\end{align*}
	Observe that the partition $V(G^{\prime})=V(K_\delta)\cup V(K_{n-(b+1)\delta+bk-1})\cup V((b\delta-bk+1)K_1)$ is equitable, by Lemma \ref{le:2.2.}, the largest root of $g_{Q_{D}(G^{\prime})}(x)=0$ equals $\eta_{1}(G^{\prime})$. 
	
	In order to prove $\eta_{1}(G_{3})>\eta_{1}(G^{\prime})$, it suffices to show $g_{Q_{D}(G_3)}(\eta_{1}(G^{\prime}))<0$. Note that $g_{Q_{D}(G^{\prime})}(\eta_{1}(G^{\prime}))=0$. Hence,
	\begin{align*}
		g_{Q_{D}(G_3)}(\eta_1(G^{\prime}))=&g_{Q_{D}(G_3)}(\eta_1(G^{\prime}))-g_{Q_{D}(G^{\prime})}(\eta_1(G^{\prime}))
		\\=&(s-\delta)[ (2bs-2bk-2b+1)\eta_{1}^{2}(G^{\prime})+((2bk+2b+1)n+4b^2s^3-(4b^2\delta
		\\&+8b^2k+8b^2-4b)s^2+((8b^2k-8b)\delta-2bn+4b^2k^2+(16b^2-4b)k+4b^2
		\\&-4b)s+(-4b^2k^2+8bk+4b^2+4b-4)\delta-8b^2k^2+(-8b^2+4b)k)\eta_{1}(G^{\prime})
		\\&-2b^2s^4+(-4b^2n+2b^2\delta+4b^2k+12b^2-2b)s^3+((4b^2\delta+8b^2k+8b^2
		\\&-4b) n+(-4b^2k-8b^2+2b)\delta-2b^2k^2+(-24b^2+2b)k-18b^2+12b)s^2
		\\&+(((-8b^2k+8b)\delta-4b^2k^2+(-16b^2+4b)k-4b^2+8b)n+(2b^2k^2
		\\&+(16b^2-2b)k-2b^2-16b)\delta+12b^2k^2+(36b^2-12b)k+8b^2-18b)s
		\\&-2n^2+((4b^2k^2-8bk-4b^2-4b+4)\delta+8b^2k^2+(8b^2-8b)k-4b+6)n
		\\&-18b^2k^2-(16b^2-18b)k-2b^2\delta^2+(-8b^2k^2+(4b^2+16b)k+8b^2+6b
		\\&-8)\delta+8b-4].
	\end{align*}
	Let $z(x)=(2bs-2bk-2b+1)x^{2}+[(2bk+2b+1)n+4b^2s^3-(4b^2\delta+8b^2k+8b^2-4b)s^2+((8b^2k-8b)\delta-2bn+4b^2k^2+(16b^2-4b)k+4b^2-4b)s+(-4b^2k^2+8bk+4b^2+4b-4)\delta-8b^2k^2+(-8b^2+4b)k]x-2b^2s^4+(-4b^2n+2b^2\delta+4b^2k+12b^2-2b)s^3+[(4b^2\delta+8b^2k+8b^2-4b)n+(-4b^2k-8b^2+2b)\delta-2b^2k^2+(-24b^2+2b)k-18b^2+12b]s^2+[((-8b^2k+8b)\delta-4b^2k^2+(-16b^2+4b)k-4b^2+8b)n+(2b^2k^2+(16b^2-2b)k-2b^2-16b)\delta+12b^2k^2+(36b^2-12b)k+8b^2-18b]s-2n^2+[(4b^2k^2-8bk-4b^2-4b+4)\delta+8b^2k^2+(8b^2-8b)k-4b+6]n
	-18b^2k^2-(16b^2-18b)k-2b^2\delta^2+[-8b^2k^2+(4b^2+16b)k+8b^2+6b
	-8]\delta+8b-4$
	be a real function in $x$ with $x\in[2n+4b\delta-4bk+1,+\infty)$. Then the axis of symmetry of $z(x)$ is
	\begin{align*}
		x=&-\frac{(2bk+2b+1)n+4b^2s^3-(4b^2\delta+8b^2k+8b^2-4b)s^2+((8b^2k-8b)\delta-2bn+4b^2k^2}{2(2bs-2bk-2b+1)}
		\\&\dfrac{+(16b^2-4b)k+4b^2-4b)s+(-4b^2k^2+8bk+4b^2+4b-4)\delta-8b^2k^2+(-8b^2+4b)k}{2(2bs-2bk-2b+1)}.
	\end{align*}
	Since $k+1\leq s \leq\delta -1$ and $n\geq\frac{6}{5}b^2\delta^3+\frac{8}{5}kb^2\delta^2$, we have
	\begin{align*}
		&2(2bs-2bk-2b+1)(2n+4b\delta-4bk+1)+[(2bk+2b+1)n+4b^2s^3-(4b^2\delta+8b^2k+8b^2
		\\&-4b)s^2+((8b^2k-8b)\delta-2bn+4b^2k^2+(16b^2-4b)k+4b^2-4b)s+(-4b^2k^2+8bk+4b^2
		\\&+4b-4)\delta-8b^2k^2+(-8b^2+4b)k]
		\\=&4b^2s^3-(4b^2\delta+8b^2k+8b^2-4b)s^2+[(8b^2k+16b^2-8b)\delta+6bn+4b^2k^2-4bk+4b^2]s+(5
		\\&-6b-6bk)n-[4b^2k^2+(16b^2-8b)k+12b^2-12b+4]\delta+8b^2k^2+(8b^2-8b)k-4b+2
		\\\geq&4b^2(k+1)^3-(4b^2\delta+8b^2k+8b^2-4b)(\delta-1)^2+[(8b^2k+16b^2-8b)\delta+6bn+4b^2k^2-4bk		\\&+4b^2](k+1)+(5-6b-6bk)n-[4b^2k^2+(16b^2-8b)k+12b^2-12b+4]\delta+8b^2k^2
		\\&+(8b^2-8b)k-4b+2
		\\=&5n-4b^2\delta^3-(8b^2k-4b)\delta^2+(4b^2k^2+24b^2k+16b^2-4b-4)\delta+8b^2k^3+(24b^2-4b)k^2
		\\&+(16b^2-12b)k+2
		\\\geq&5(\frac{6}{5}b^2\delta^3+\frac{8}{5}kb^2\delta^2)-4b^2\delta^3-(8b^2k-4b)\delta^2+(4b^2k^2+24b^2k+16b^2-4b-4)\delta+8b^2k^3
		\\&+(24b^2-4b)k^2+(16b^2-12b)k+2
		\\=&2b^2\delta^3+4b\delta^2+(4b^2k^2+24b^2k+16b^2-4b-4)\delta+8b^2k^3+(24b^2-4b)k^2+(16b^2-12b)k
		\\&+2
		\\>&0.
	\end{align*}
	Consequently, $z(x)$ is monotonically increasing in the interval $[2n+4b\delta-4bk+1,+\infty).$ We get
	\begin{align*}
		z(\eta_1(G^{\prime}))\geq& z(2n+4b\delta-4bk+1)
		\\=&-2b^2s^4+[4b^2n+(16b^3+2b^2)\delta+(4b^2-16b^3)k+16b^2-2b]s^3-[(4b^2\delta+8b^2k
		\\&+8b^2-4b)n+16b^3b^2+((16b^3+4b^2)k+32b^3-4b^2-2b)\delta-(32b^3-2b^2)k^2
		\\&-(32b^3-48b^2+2b)k+26b^2-16b]s^2+[4bn^2+((8b^2k+24b^2-8b)\delta+4b^2k^2
		\\&-(8b^2+4b)k+4b^2+6b)n+(32b^3k+32b^3-32b^2)\delta^2+((-16b^3+2b^2)k^2
		\\&+(40b^2-2b)k+16b^3-2b^2-24b)\delta-16b^3k^3+(32b^2-32b^3)k^2-(16b^3-52b^2
		\\&+12b^2-20b]s-(4bk+4b-4)n^2+[(-4b^2k^2+(-24b^2+8b)k-20b^2+24b-4)
		\\&\cdot\delta+16b^2k^2+(16b^2-26b)k-10b+11]n-[16b^3k^2+(32b^3-32b^2)k+16b^3
		\\&-30b^2+16b]\delta^2+[16b^3k^3+(32b^3-44b^2)k^2+(16b^3-44b^2+40b)k-4b^2+18b
		\\&-12]\delta-10b^2k^2+(-8b^2+12b)k+6b-3
		\\\geq&-2b^2(\delta-1)^4+[4b^2n+(16b^3+2b^2)\delta+(4b^2-16b^3)k+16b^2-2b](k+1)^3
		\\&-[(4b^2\delta+8b^2k+8b^2-4b)n+16b^3b^2+((16b^3+4b^2)k+32b^3-4b^2-2b)\delta
		\\&-(32b^3-2b^2)k^2-(32b^3-48b^2+2b)k+26b^2-16b](\delta-1)^2+[4bn^2+((8b^2k
		\\&+24b^2-8b)\delta+4b^2k^2-(8b^2+4b)k+4b^2+6b)n+(32b^3k+32b^3-32b^2)\delta^2
		\\&+((-16b^3+2b^2)k^2+(40b^2-2b)k+16b^3-2b^2-24b)\delta-16b^3k^3+(32b^2-32b^3)
		\\&\cdot k^2-(16b^3-52b^2+16b)k+12b^2-20b](k+1)-(4bk+4b-4)n^2+[(-4b^2k^2
		\\&+(-24b^2+8b)k-20b^2+24b-4)\delta+16b^2k^2+(16b^2-26b)k-10b+11]n
		\\&-[16b^3k^2+(32b^3-32b^2)k+16b^3-30b^2+16b]\delta^2+[16b^3k^3+(32b^3-44b^2)k^2
		\\&+(16b^3-44b^2+40b)k-4b^2+18b-12]\delta-10b^2k^2+(-8b^2+12b)k+6b-3
		\\=&4n^2+[-4b^2\delta^3+(-8b^2k+4b)\delta^2+(4b^2k^2+24b^2k+16b^2+8b-4)\delta+8b^2k^3
		\\&+(24b^2-4b)k^2+(16b^2-24b)k+11]n-(16b^3+2b^2)\delta^4-[(16b^3+4b^2)k-12b^2
		\\&-2b]\delta^3+[(48b^3-2b^2)k^2+(96b^3-40b^2+2b)k+64b^3-48b^2-4b]\delta^2+[(16b^3
		\\&+4b^2)k^3+(8b^2-2b)k^2+(92b^2+10b)k+60b^2-36b-12]\delta+(-32b^3+4b^2)k^4
		\\&+(-96b^3+60b^2-2b)k^3+(-64b^3+132b^2-22b)k^2+(60b^2-28b)k-3.
	\end{align*}
	Let $z_{1}(x)=4x^2+[-4b^2\delta^3+(-8b^2k+4b)\delta^2+(4b^2k^2+24b^2k+16b^2+8b-4)\delta+8b^2k^3+(24b^2-4b)k^2+(16b^2-24b)k+11]x-(16b^3+2b^2)\delta^4-[(16b^3+4b^2)k-12b^2-2b]\delta^3+[(48b^3-2b^2)k^2+(96b^3-40b^2+2b)k+64b^3-48b^2-4b]\delta^2+[(16b^3+4b^2)k^3+(8b^2-2b)k^2+(92b^2+10b)k+60b^2-36b-12]\delta+(-32b^3+4b^2)k^4+(-96b^3+60b^2-2b)k^3+(-64b^3+132b^2-22b)k^2+(60b^2-28b)k-3$ be a real function in $x$ with $x\in[\frac{6}{5}b^2\delta^3+\frac{8}{5}kb^2\delta^2,+\infty).$
	We may obtain the derivative function of $z_1(x)$ is
	\begin{align*}
		z_{1}^{\prime}(x)=&8x-4b^2\delta^3+(-8b^2k+4b)\delta^2+(4b^2k^2+24b^2k+16b^2+8b-4)\delta+8b^2k^3
		\\&+(24b^2-4b)k^2+(16b^2-24b)k+11
		\\\geq&8(\frac{6}{5}b^2\delta^3+\frac{8}{5}kb^2\delta^2)-4b^2\delta^3+(-8b^2k+4b)\delta^2+(4b^2k^2+24b^2k+16b^2+8b-4)\delta
		\\&+8b^2k^3+(24b^2-4b)k^2+(16b^2-24b)k+11
		\\=&\frac{28}{5}b^2\delta^3+(\frac{24}{5}kb^2+4b)\delta^2+(4b^2k^2+24b^2k+16b^2+8b-4)\delta+8b^2k^3+(24b^2-4b)k^2
		\\&+(16b^2-24b)k+11
		\\\geq&\frac{1}{5}(112k^3+392k^2+388k+108)b^2-(8k-12)b-4k+7
		\\\geq&\frac{1}{5}(112k^5+392k^4+388k^3+68k^2+40k+35)
		\\>&0.
	\end{align*}
	Therefore,  $z_{1}(x)$ is montonically increasing in the interval $[\frac{6}{5}b^2\delta^3+\frac{8}{5}kb^2\delta^2,+\infty)$. We get
	\begin{align*}
		z_1(n)\geq&z_1(\frac{6}{5}b^2\delta^3+\frac{8}{5}kb^2\delta^2)
		\\=&\frac{1}{25}[24b^4\delta^6+(120b^3-16kb^4)\delta^5+(56b^4k^2+(720b^4+160b^3)k+480b^4-160b^3-170b^2)
		\\&\cdot\delta^4+(400b^4k^3+(1680b^4-120b^3)k^2+(1120b^4-800b^3-260b^2)k+630b^2+50b)\delta^3
		\\&+(320b^4k^4+(960b^4-160b^3)k^3+(640b^4+240b^3-50b^2)k^2+(2400b^3-560b^2+50b)
		\\&\cdot k+1600b^3-1200b^2-100b)\delta^2+((400b^3+100b^2)k^3+(200b^2-50b)k^2+(2300b^2
		\\&+250b)k+1500b^2-900b-300)\delta+(-800b^3+100b^2)k^4-(2400b^3-1500b^2+50b)k^3
		\\&-(1600b^3-3300b^2+550b)k^2+(1500b^2-700b)k-75]
		\\\geq&\frac{1}{25}[(784k^6+5488k^5+14136k^4+17824k^3+11776k^2+3888k+504)b^4-(560k^4
		\\&+520k^3-3400k^2-4920k-1560)b^3-(280k^4-310k^3-3520k^2-3290k-760)b^2
		\\&-(200k^2+1350k+950)b-300k-375]
		\\>&0.
	\end{align*}

	Therefore, we can conclude that $z(\eta_1(G^{\prime}))>0$ and $g_{Q_{D}(G_3)}(\eta_1(G^{\prime}))<0$. Hence, $\eta_1(G_3)>\eta_{1}(G^{\prime})$. Combining $\eta_1(G)\geq\eta_1(G_1)\geq\eta_1(G_3)$, we may conclude that $\eta_1(G)>\eta_1(G^{\prime})$, a contradiction to the condition.
	
	This completes the proof.
	$\hfill\square$\\

	\textbf{Declaration of competing interest}
	
	The authors declare that they have no known competing financial interests or personal relationships that could have appeared to influence the work reported in this paper.

	\textbf{Date availability}
	
	No date was used for the research described in the article.

\end{document}